\documentclass[11pt,a4paper]{article}
\usepackage{amsfonts,amsbsy,amssymb,amsmath,graphicx}
\usepackage{epsfig}
\usepackage{dsfont}
\usepackage{mathrsfs}
\usepackage{geometry}
\usepackage[T1]{fontenc}
\usepackage{subeqnarray}
\usepackage{enumerate}
\usepackage{amsbsy}
\usepackage{bm}
\usepackage{color}
\usepackage{rotating}
\usepackage{epsfig}
\usepackage{ifpdf}
\usepackage{cite}
\usepackage{setspace}
\usepackage[colorlinks=true,linkcolor=blue,citecolor=blue]{hyperref}

\newenvironment{proof} {\noindent {\em \textbf{Proof}} } { \hfill \fbox{~} \\ }

\def\1{1\kern-.20em {\rm l}}
\def\abs  #1\par{\vskip2truemm {\noindent} {\parindent=12truemm\narrower\baselineskip=3truemm \pf #1\par}}

\newtheorem{theorem}{Theorem}[section]

\newtheorem{corollary}[theorem]{Corollary}

\newtheorem{definition}[theorem]{Definition}

\newtheorem{lemma}[theorem]{Lemma}

\newtheorem{proposition}[theorem]{Proposition}

\numberwithin{equation}{section}

\def\1{1\kern-.20em {\rm l}}
\def\qed{\ifmmode\mbox{\hfill\sqb}\else{\ifhmode\unskip\fi
\nobreak\hfil
\penalty50\hskip1em\null\nobreak\hfil\sqb
\parfillskip=0pt\finalhyphendemerits=0\endgraf}\fi}
\def\cqfd{\ifmmode\sqw\else{\ifhmode\unskip\fi\nobreak\hfil
\penalty50\hskip1em\null\nobreak\hfil\sqw
\parfillskip=0pt\finalhyphendemerits=0\endgraf}\fi}
\title{\bf On a new class of 2-orthogonal polynomials, II :\\
 The integral representations}
 \author{\small Khalfa DOUAK\footnote{The corresponding author. } $\ $ and Pascal MARONI$^{\dag}$\\
\small Laboratoire Jacques-Louis Lions, Sorbonne  Universit\'{e}-CNRS,\\
\small Boite courrier 187, 75252, Paris Cedex 05, France\\
\small Email: $^{\ast}$khalfa.douak@gmail.com {\rm and} $^{\dag}$pascalmaroni@orange.fr}
\date{\today}
\begin{document}
\maketitle
\begin{abstract}
\noindent A new class of 2-orthogonal polynomials satisfying orthogonality conditions with respect  to a pair of linear functionals $(u_0,u_1)$ was presented in Douak K \& Maroni P [On a new class of 2-orthogonal polynomials, I: the recurrence relations and some properties. Integral Transforms Spec Funct. 2021;32(2):134-153]. Six worthwhile special cases were pointed out there. Here we are precisely interesting with  the integral representation problem for the functionals associated to these polynomials in  each case.
The focus will be on the  matrix  differential equation $\big({\bf\Phi U}\big)'+{\bf\Psi U}=0$, where ${\bf\Phi}$, ${\bf\Psi}$ are  $2\times2$ polynomial matrices and  ${\bf U}$ stands for the vector ${^t}(u_0 , u_1)$. We first  establish the differential systems satisfied by the functionals $u_0$ and $u_1$ and  then, depending on the case, we show that they admit integral representation via weight functions supported on the real line or on positive real line and defined in terms of various special functions. In order for certain integral representations to exist, addition of Dirac mass is necessary.
\end{abstract}


%
%


\maketitle


\section{Introduction and Motivation}	
The present paper represents a continuation of our earlier  work \cite{1} where we gave a new class of classical 2-orthogonal polynomials,
 we have achieved it for a matter  of completeness.  The reader is therefore invited to take note of the principal outcomes presented therein before reading this part.
Our main purpose  here is to  look for an integral representation  of both functionals $u_0$ and $u_1$ with respect to which the 2-orthogonality conditions  hold. The primary  objective is to seek  a pair of weight functions $w_0$ and $w_1$ supported on (intervals of) the real axis.
These weights turn out to  satisfy specific linear differential equations of order less than or equal to $2$, and subject to certain boundary conditions.
In discussing these equations and in trying to solve them, we will restrict ourselves to imposing some conditions on the free parameters to find the desired solutions.
Recall in passing that, there are many $2$-orthogonal polynomials for which the associated weights are not known  or are not unique but has a characteristic third order standard recurrence relation.\par
\noindent  For the convenience of the reader we follow the same notations used in Part I.\\
  By ${\mathscr P}$ we denote the vector space of polynomials of one variable with complex coefficients and we let ${\mathscr P}'$ its algebraic dual. The duality brackets between ${\mathscr P}$ and  ${\mathscr P}'$ is denoted by $\bigl< .\,,\, .\bigr>$. For  a sequence of monic polynomials $\{P_n\}_{n \geqslant 0}$, with $\deg P_n=n$, we associate its   dual sequence $\{u_n\}_{n \geqslant 0}$  defined by $\bigl< u_n , P_n \bigr> = \delta_{nm} ;\, n, m \geqslant0$, where $\delta_{n m}$ is the Kronecker's delta symbol.\\
Let us now introduce the  2-orthogonal polynomials sequence (2-OPS) $\{P_n\}_{n\geqslant0}$ previously investigated in  \cite{1}.
These polynomials have been obtained by working only on the system satisfied by their recurrence coefficients whose the expressions
  are explicitly determined. They satisfy the third-order recurrence relation
\begin{align}
&P_{n+3}(x) = (x - \beta_{n+2}) P_{n+2}(x) - \alpha_{n+2} P_{n+1}(x) -\gamma_{n+1} P_n (x) \, , \; n \geqslant 0 ,\label{eq:1.1}\\
&P_0(x) = 1 \, , \; P_1(x) = x - \beta_0 \ ,  \; P_2(x) = (x -\beta_1)P_1(x) - \alpha_1,\label{eq:1.2}
\end{align}
with the recurrence coefficients are given, for $n\geqslant 0$, by
\begin{align}
\beta_{n} &= (\varepsilon_nr-2s)n-\frac{1}{2}(1-\varepsilon_n)r+ \beta_0,\label{eq:1.3}\\
\alpha_{n+1}&= (n+1)\Bigl[  (r^2+s^2)  n+(1-\varepsilon_n)rs +\alpha_1\Bigr],\label{eq:1.4}\\
\gamma_{n+1}&= (n+2)(n+1)\Bigl[\varepsilon_n (r-\varepsilon_ns )^2 r n- \frac{1}{2}(1-\varepsilon_n)\big((r+s)^2+2\alpha_1 \big)r +\gamma\Bigr].\label{eq:1.5}
\end{align}
A remarkable character of these polynomials is that they possess the Hahn’s property \cite[Def. 1.8]{1} and so they are considered to be classical.
 This means that the polynomials $Q_n:=(n+1)^{-1}P'_{n+1}, \, n=0, 1, 2, \ldots,$ are also 2-orthogonal.
 In this regard, we have established that they satisfy the third-order recurrence relation
\begin{align}
&Q_{n+3}(x) = (x - \tilde\beta_{n+2})Q_{n+2}(x)  - \tilde\alpha_{n+2}Q_{n+1}(x) - \tilde\gamma_{n+1} Q_n(x)\,  , \; n \geqslant 0 ,\label{eq:1.6}\\
&Q_0(x) = 1 \, , \; Q_1(x) = x - \tilde\beta_0 \; , \;\; Q_2(x) = (x -\tilde\beta_1)Q_1(x) - \tilde\alpha_1,\label{eq:1.7}
\end{align}
with
\begin{align}
\tilde\beta_{n}&= -(\varepsilon_nr+2s)n- \frac{1}{2}(1+\varepsilon_n)r-s+ \beta_0, \; n \geqslant 0,\label{eq:1.8}\\
\tilde\alpha_{n}&= n\Bigl[(r^2+s^2)n+(1-\varepsilon_n)rs+\alpha_1\Bigr],\; n \geqslant 1,\label{eq:1.9}\\
\tilde\gamma_{n}&= (n+1)n\Bigl[\varepsilon_n (r-\varepsilon_ns )^2 r n- \frac{1}{2}(1-\varepsilon_n)\big((r+s)^2+2\alpha_1 \big)r +\gamma\Bigr],\;
n\geqslant 1.\label{eq:1.10}
\end{align}
For this solution, we can see that recurrence coefficients of both $\{P_n\}_{n \geqslant 0}$ and $\{Q_n\}_{n \geqslant 0}$ being interconnected in pairs by
\begin{align}
\tilde\beta_n = \beta_{n+1} + \delta_n,\ n\geqslant0\ ;\
\tilde\alpha_n = \frac{n}{n+1}\alpha_{n+1}\ ; \ \tilde\gamma_n = \frac{n}{n+2}\gamma_{n+1},\ n\geqslant1,\label{eq:1.11}
\end{align}
where $\varepsilon_n=(-1)^n$, $\delta_n=s+r\varepsilon_n,\ n\geqslant0$, and $s$, $r$ are two arbitrary constants, with the notation $\gamma:=\frac{1}{2}\gamma_1$.
 Recall that the last two identities in \eqref{eq:1.11} follow, respectively, from formulas (2.12)-(2.13)  given in \cite[p.7]{1} if we take $\rho_n=1$, $\theta_n=1$. It is understood that, for all $n\geqslant1$, $\alpha_n\ne0$ and  $\gamma_{n}\ne0$.\par

\noindent In addition to the forgoing, we have the differential-recurrence relation
\begin{align}
P_{n}(x) = Q_{n}(x) - \delta_{n+1} n Q_{n-1}(x) ,\ n\geqslant 0\ ; (Q_{-1}=0).\label{eq:1.12}
\end{align}
It is worth noting here that the above expression of the $\tilde\beta_{n}$'s coefficients  is slightly different from the one given in \cite{1}.
Roughly  speaking, \eqref{eq:1.8}   corrects and replaces the identity (2.32) established there without compromising any of the obtained results.
It we consider now the dual sequence $\{v_n\}_{n \geqslant 0}$ associated to $\{Q_n\}_{n \geqslant 0}$, then the two sequences $\{u_n\}_{n \geqslant 0}$
 and $\{v_n\}_{n \geqslant 0}$ are interconnected via the identity \cite{2}
\begin{align}
v'_n=-(n+1)u_{n+1},\ n\geqslant0.\label{eq:1.13}
\end{align}
By the result given in \cite[Prop. 2.3]{3}, the following recurrence relation holds  in ${\mathscr P}'$
\begin{align}
  xu_n= u_{n-1}+\beta_nu_{n}+\alpha_{n+1}u_{n+1}+\gamma_{n+1}u_{n+2},\ n\geqslant0\ ;\ (u_{-1}=0),\label{eq:1.14}
\end{align}
It was also shown that the above identity  permits to write the elements of the dual sequence $\{u_n\}_{n \geqslant 0}$ in terms of the pair $(u_0,u_1)$ with polynomial coefficients of degree less than or equal to $n$.\\
 In fact the recurrence   \eqref{eq:1.14}  is another property characterizing the 2-OPS $\{P_n\}_{n \geqslant 0}$.\\
Let ${\bf U}={^t}(u_0 , u_1)$ be the vector functional with respect to (w.r.t.) which the polynomials  $P_n,\, n=0, 1, \ldots$, are  2-orthogonal and ${\bf V}={^t}(v_0 , v_1)$ be the vector  functional w.r.t. which the polynomials $Q_n,\, n=0, 1, \ldots$, are 2-orthogonal.
Due to the fact that the polynomials  $P_n$ are classical, it was shown in \cite[Sec. 3]{4} that the vector ${\bf U}$ satisfies the  matrix  differential equation $\big({\bf\Phi U}\big)'+{\bf\Psi U}=0$, where ${\bf\Phi}$ and ${\bf\Psi}$ are  $2\times2$ polynomial matrices. Further,
  ${\bf V}$ and ${\bf U}$ are interlinked via the relation ${\bf V}={\bf \Phi}{\bf U}$. \\

  The structure of the paper is  as follows. In the next section we remind some special functions and useful operations we need in the sequel.
Section 3 is devoted to the fundamental characterization of the classical 2-orthogonal polynomials via the matrix differential  equation (Theorem 3.1).
In  the fourth section, we first expose the outline of the procedure used to look for integral representation for the pair of functionals $u_0$ and $u_1$ via weight functions denoted $w_0$ and $w_1$, respectively. Afterwards, we present in  each  special case our  results  by solving the differential equations satisfied by these  functions. In that regard, depending on the case and under certain conditions, we show that the obtained 2-OPS are associated with  weights involving certain special functions as  Airy function, Gauss-Airy function, modified Bessel function or Tricomi function.
In three subcases, however, we  obtain that the resulting polynomials are 2-orthogonal with respect to  a pair of weights $(w_0,w_1)$, where $w_0$ is either Laguerre's weight or Hermite's one  and $w_1$ is given as a combination of $w_0$ and incomplete gamma function or error function. Note finally that certain weight functions are supported on the whole real line, while others are supported on the positive real line (by adding Dirac masses in some subcases).

\section{Some Special Functions and Useful Operations}

In this section we have enclosed a brief exposition of some special functions with certain of their main properties that will be used later. We have also recalled certain useful operations in  ${\mathscr P}'$.

\subsection{Modified Bessel functions}

 For  $x>0$ and $\nu>-1$, the modified Bessel functions $I_\nu(x)$ of the first kind and $K_\nu(x)$ of the second kind (Macdonald function) are continuous  positive functions of $x$ and $\nu$ (see for instance \cite[Sec. 10.25]{5}). Define  the two scaled modified Bessel functions  $\omega_\nu(x)$ and $\rho_\nu(x)$   as follows
 \begin{align}
 \hskip1cm\omega_\nu(x)&=x^{\nu/2}I_\nu(2\sqrt{x}), && x>0 \, ,\ \nu>-1,\label{eq:2.1}\\
\rho_\nu(x)&=2x^{\nu/2}K_\nu(2\sqrt{x}), && x>0 \, ,\ \nu>-1. \label{eq:2.2}\hskip2cm
\end{align}
Based on the properties of $I_\nu(x)$ and $K_\nu(x)$, we obtain that  both  $\omega_\nu(x)$ and $\rho_\nu(x)$ are also positive over the positive real axis  satisfying
 some differential properties and recurrence relations  (see for instance \cite{6,7,8,9}).

 \subsection{Airy functions}

The Airy functions of the first and second kind denoted ${\rm Ai}(z)$ and ${\rm Bi}(z)$, respectively, are two standard solutions linearly independent of the Airy differential equation \cite[Chap. 9]{5}:
\begin{align}
 y''-zy=0.\label{eq:2.3}
 \end{align}
The general solution to \eqref{eq:2.3} is then  $y=c_1{\rm Ai}(z)+c_2{\rm Bi}(z),$ where $c_1$ and $c_2$ are arbitrary constants.\\
By introducing the auxiliary argument $\zeta:={\frac{2}{3}}z^{3/2}$ with $|\arg z|<{\frac{2\pi}{3}}$, the Airy function ${\rm Ai}(z)$ and its derivative ${\rm Ai}'(z)$ may be expressed in terms of the Macdonald function as follows
\begin{align}
{\rm Ai}(z)={\frac{1}{\pi}}\sqrt{\frac{z}{3}}K_{\pm\frac{1}{3}}(\zeta)\quad\mbox{and}\quad
{\rm Ai}'(z)=-{\frac{1}{\pi}}{\frac{z}{\sqrt{3}}}K_{\pm\frac{2}{3}}(\zeta).\label{eq:2.4}
\end{align}
The original Airy function of the first kind  with a real argument is defined by the Airy integral
\begin{align}
{\rm Ai}(x)=\frac{1}{\pi}\int_0^{+\infty}\!\!\!\cos\Big(xt+\frac{t^3}{3}\Big)dt ,\; x\geqslant0.\label{eq:2.5}
\end{align}
A various of integral representations of the function ${\rm Ai}(x)$ have been given by different authors (see, e.g., \cite{10} and the references therein).
 Some extensions of this function was also established afterwards. In this sense, we draw attention to two examples. First, the two-variable extension introduced in \cite{11} as
\begin{align}
{\rm Ai}(x,\tau)=\frac{1}{2\pi}\int_{-\infty}^{+\infty}\!\!\! \exp\Big[i\Big(xt+\tau t^3\Big)\Big]dt \ , \;  \tau>0. \label{eq:2.6}
\end{align}
The function ${\rm Ai}(x,\tau)$ satisfies the differential equation $3 \tau y''-xy=0$, as is easy to check. When $\tau$ is interpreted as a parameter, ${\rm Ai}(x,\tau)$ is expressed in terms of the ordinary Airy function as follows
\begin{align}
{\rm Ai}(x,\tau)={(3\tau)}^{-\frac{1}{3}}{\rm Ai}\Big({(3\tau)}^{-\frac{1}{3}}x\Big).\label{eq:2.7}
\end{align}
 Another function called Gauss-Airy function and presented in \cite{12} as a three-variable extension of the ordinary Airy function is defined by introducing a Gaussian distribution in the integral representation of the Airy function, that is,
\begin{align}
{\rm GAi}(x,\xi,\tau)=\frac{1}{\pi}\int_{0}^{+\infty}\!\! e^{-\xi t^2}\cos\Big(xt+\tau\frac{t^3}{3}\Big)dt\ ,\; \xi\geqslant 0 ,\ x, \tau \in\mathbb{R},\label{eq:2.8}
\end{align}
where $\xi$ and $\tau$ may be considered as two parameters. It is also shown that this function generates the so-called three-variable Hermite polynomials presented by the following integral representation as
\begin{align}
{_3}H_n(x,y,z)=\int_{-\infty}^{+\infty}\! t^n{\rm GAi}(t-x,y,z)dt\ ,\ \ n\geqslant 0.\label{eq:2.9}
\end{align}
 For more properties of these polynomials  we refer the reader to \cite{12} and the references given there.\\
Besides, replacing $a$ by $ia$ into the formula (2.25) given in \cite[p.10]{10} yields
\begin{align}
\frac{1}{\pi}\int_{0}^{+\infty}\!\!\! e^{-a t^2} \cos\Big(xt+\frac{t^3}{3}\Big)dt= e^{a(x+\frac{2}{3}a^2)}{\rm Ai}(x+a^2).\label{eq:2.10}
\end{align}
With specific values of $\xi$ and $\tau$, the use of {\eqref{eq:2.8}} with  {\eqref{eq:2.6}} and \eqref{eq:2.10} respectively gives
\begin{align}
{\rm GAi}(x,0,3\tau)={\rm Ai}(x,\tau)\quad \mbox{and}\quad{\rm GAi}(x,a,1)= e^{a(x+\frac{2}{3}a^2)}{\rm Ai}(x+a^2).\label{eq:2.11}
\end{align}
 A trivial verification shows that the  function ${\rm Ai}(x+a^2)$ is a solution of the Airy-type equation
\begin
{align}y''-(x+a^2)y=0.\label{eq:2.12}
\end{align}
We finish with the asymptotic approximations for ${\rm Ai}(\pm x)$ and ${\rm Ai}'(\pm x)$, as $x\to+\infty$ \cite[Sec. 9.7]{5}:
 \begin{align}
{\rm Ai}(x)\sim\frac{1}{2\sqrt{\pi}}x^{-1/4}e^{-\zeta} \quad &\mbox{and}\quad
{\rm Ai}'(x)\sim -\frac{1}{2\sqrt{\pi}}{x^{1/4}e^{-\zeta}};\label{eq:2.13}\\
{\rm Ai}(-x)\sim\frac{1}{\sqrt{\pi}}x^{-1/4}\cos\Big(\zeta-\frac{\pi}{4}\Big)\quad &\mbox{and}\quad
{\rm Ai}'(-x)\sim \frac{1}{\sqrt{\pi}}x^{1/4}\sin\Big(\zeta-\frac{\pi}{4}\Big).\label{eq:2.14}
\end{align}

\subsection{Kummer functions}

 The confluent hypergeometric function $M(a\, ; c\, ;z)$ designated also by the symbol ${_1}F_1(a\, ; c\, ; z)$, known as Kummer function of the first kind, is defined as
 \cite[Sec. 13]{5}
\begin{align}
M(a\, ; c\, ;z) =\sum_{n=0}^\infty{\frac{(a)_n}{(c)_n}}{\frac{z^n}{n!}}\; ,\quad |z|<+\infty\ ,\; (c \ne 0, -1, -2, \ldots),\label{eq:2.15}
\end{align}
where $z$ is a complex number, $a$ and $c$ are parameters which can take arbitrary real or complex values, with $(\alpha)_n$ is  Pochhammer's symbol defined, for any  $\alpha$ real or complex, by $(1)_n = n!$, $(\alpha)_0 = 1$ and $(\alpha)_n= \alpha(\alpha + 1) \cdots (\alpha + n-1),\ n= 1, 2,\ldots$.\\
This function is the first two standard solutions of the following confluent hypergeometric equation, also known as Kummer's equation,
\begin{align}
 zy''+(c-z)y'-ay=0.\label{eq:2.16}
 \end{align}
Another standard solution of Eq. \eqref{eq:2.16} closely related to $M$ is the function $U(a\, ; c\, ; z)$ (introduced by Tricomi) which is determined uniquely
 by the property
\begin{align}
U(a\, ; c\, ; z) \sim z^{-a}, \quad z\to +\infty \quad \hbox{in} \quad |\arg(z)|< \frac{3}{2}\pi.\label{eq:2.17}
\end{align}
 A second linearly independent solution to the function $U(a\, ; c\, ;z)$ is constructed as $ e^zU(c-a\, ; c\, ; -z)$.
 This fundamental pair of solutions is known to be numerically satisfactory in the neighborhood of infinity.\\
For small positive values of the argument, the value of the Tricomi function depends on the $c$ parameter.
The tables below  bring together some values of $U(a\, ; c\, ; x)$, as $x\to0$, when  $a, c$ are reals and $b=1+a-c$ is an auxiliary parameter. Here, the notations are that of \cite{13}
with  $\bm{\psi}$ is the Psi function and $\bm{\gamma}$ is the Euler constant.\\

\noindent \begin{table}[th]
{\begin{tabular}{@{}|c|c|c|@{}}\hline
 $ c=0 $ & $ 0<c<1 $ & $ c=1 $ \\ \hline
 $\frac{1}{\Gamma(1+a)}+\frac{x\ln{x}}{\Gamma(a)}$ &$ \frac{\Gamma(1 - c)}{\Gamma(b)}+\frac{\Gamma(c - 1)}{\Gamma(a)}x^{1-c}$ &$-\frac{2\bm{\gamma}+\bm{\psi}(a)+\ln{x}}{\Gamma(a)}$\\ \hline
\end{tabular}}
\end{table}
{Table 1. Values of $U$ when $x$ is small and $0\leqslant c\leqslant1$.}\\

\noindent \begin{table}[th]
{\begin{tabular}{|c|c|c|}\hline
   $ 1<c< 2  $ & $ c=2 $ & $ c>2 $\\ \hline
 $\frac{\Gamma(1 - c)}{\Gamma(b)}+\frac{\Gamma(c - 1)}{\Gamma(a)x^{c-1}}$ & $\!\frac{1}{\Gamma(a)x}+\frac{\ln{x}}{\Gamma(a-1)}$
&$ \frac{\Gamma(c - 2)[c-2-bx]}{\Gamma(a)x^{c-1}}$ \\ \hline
\end{tabular}}
\end{table}
{Table 2. Values of $U$ when $x$ is small and $c>1$.}\\

\noindent We also give the following useful formulas that we need later \cite{5}.\\
 $\bullet$ Two differentiation formulas of $U$:
   \begin{align}
 U'(a\, ; c\, ; z) &=-a U(a+1\, ; c+1\, ; z),\label{eq:2.18}\\
 \frac{d}{dz}\left(e^{-z}U(a\, ; c\, ; z)\right)&=- e^{-z}U(a\, ; c+1\, ; z).\label{eq:2.19}
 \end{align}
$\bullet$ Integral representations:
\begin{align}
M(a\, ; c\, ; z) &= \frac{\Gamma(c)}{\Gamma(a)\Gamma(c-a)}\int_0^1 e^{zt} t^{a-1} (1-t)^{c-a-1}dt,\quad \Re{(c)}>\Re{(a)}>0,\label{eq:2.20}\\
U(a\, ; c\, ; z) &= \frac{1}{\Gamma(a)}\int_0^{+\infty}\!\! e^{-zt} t^{a-1} (1+t)^{c-a-1}dt, \quad \Re{(z)}>0 ,\ \Re{(a)}>0.\label{eq:2.21}
\end{align}
 We also need the next formula valid for $\Re(z)>0$, $\Re(b)>\mbox{max}\bigl(\Re(c)-1,0\bigr)$
\begin{align}
\int_0^{+\infty}\!\!\!e^{-zt}t^{b-1}U\!\left(a ; c ; t\right)\!dt\!=\!\frac{\Gamma(b)\Gamma(b\!-\!c\!+\!1)}{\Gamma(a\!+\!b\!-\!c\!+\!1)}z^{-b}\!{_2}F_1\Big(\!a , b\, ; a+b-c+1 ; 1\!-\!\frac{1}{z}\Big).\label{eq:2.22}
\end{align}
$\bullet$  Asymptotic behavior of $M(a\, ; c\, ; z)$ for large argument:
\begin{align}
& M(a\, ; c\, ; z)= \frac{\Gamma(c)}{\Gamma(a)} z^{a-c} e^z\big[1 + {\mathscr O}(|z|^{-1})\big], \ \Re{(z)}>0; \ a\ne0, -1, -2, \ldots,\label{eq:2.23} \\
& M(a\, ; c\, ; -z)\!=\! \frac{\Gamma(c)}{\Gamma(c\!-\!a)} z^{-a} \big[1\! +\! {\mathscr O}(|z|^{-1})\big], \, \Re{(z)}\!>0; \, c- a\ne0, -1, -2, \ldots.\label{eq:2.24}
\end{align}

\subsection{Incomplete gamma functions}

 The incomplete gamma function $\gamma(a,x)$ and its complementary incomplete gamma function $\Gamma(a,x)$ are usually defined via the integrals \cite[Sec. 8]{5}
\begin{align}
\gamma(a,x)=\int^x_0\!\! t^{a-1} e^{-t}dt\quad\mbox{and}\quad\Gamma(a,x)=\int^{+\infty}_x \!\! t^{a-1} e^{-t}dt\ , \; x\geqslant0\; ,\; a>0.\label{eq:2.25}
\end{align}
 Clearly, $\gamma(a,0)=0$ and $\Gamma(a,0)=\Gamma(a)$. Using Euler's integral for  $\Gamma(a)$ yields
\begin{align}
\gamma(a,x)+\Gamma(a,x)=\Gamma(a).\label{eq:2.26}
\end{align}
 For $c>0$, from \eqref{eq:2.25} we can also write
\begin{align}
\gamma(a,cx)=c^a\int_0^x \!\! t^{a-1} e^{-ct}dt\quad\mbox{and}\quad\Gamma(a,cx)=c^a\int^{+\infty}_x \!\! t^{a-1} e^{-ct}dt.\label{eq:2.27}
\end{align}
Likewise, we have an analogous formula to \eqref{eq:2.26}
\begin{align}
\gamma(a,cx)+\Gamma(a,cx)= \Gamma(a).\label{eq:2.28}
\end{align}
 Finally, the asymptotic series representation for the function ${\Gamma}(a,x)$ as $ x\to+\infty $ is
\begin{align}
{\Gamma}(a,x) \sim\Gamma(a) x^{a-1}e^{-x}\sum_{k=0}^{+\infty}\frac{1}{\Gamma(a-k)\ x^k} , \quad a>0.\label{eq:2.29}
\end{align}
\subsection{Error functions}
The  error function ${\rm erf}(x)$ and its complementary error function denoted ${\rm erfc}(x)$ are, respectively, given by the indefinite integrals \cite[Sec. 7]{5}
\begin{align}
{\rm erf}(x)=\frac{2}{\sqrt{\pi}}\int_0^x\!\! e^{-t^2}dt \  \; \hbox{and}\ \;
{\rm erfc}(x)=\frac{2}{\sqrt{\pi}}\int_x^{+\infty}\!\! e^{-t^2}dt ,  \; {\rm for}\ x\in \mathbb{R}.\label{eq:2.30}
\end{align}
Since both functions are interrelated via ${\rm erfc}(x)=1-{\rm erf}(x)$, all properties of ${\rm erfc}(x)$  can be derived from those of ${\rm erf}(x)$.
For example, we see that  ${\rm erf}(0)=0$, ${\rm erfc}(0)=1$ and  ${\rm erf}({+\infty})=1$, ${\rm erfc}(+\infty)=0 $.\\
For large positive argument,   ${\rm erfc}(x)$ has the following asymptotic behavior
\begin{align}
{\rm erfc}(x)\sim \frac{e^{-x^2}}{x\sqrt{\pi}}, \quad \hbox{as}\ x\to+\infty.\label{eq:2.31}
\end{align}

\subsection{Some useful operations}

At the end of this section, let us recall a definition and  some basic operations which we need below \cite{15}.\\
For a functional $u$ and for all  polynomials $h$ and $f$, the left-multiplication of $u$ by a polynomial $h$ is given by $\big<hu\, ,\,f \big>=\big<u\, ,\,hf\big>$,
the derivative $Du = u'$ is defined by $\big<u'\,,\,f\big> = -\big<u\, ,\,f'\big>$ ; so also is  $(fu)'= fu'+f'u$. We close this section with the following definition.
\begin{definition} A function $g$, not identically zero, locally integrable with rapid decay is said to be a representing of the null functional if
\begin{align}
\left<0\, ,\, f\right>=\int_{-\infty}^{+\infty}\!\!f(x) {g}(x)dx ,\; \forall f\in {{\mathscr P}}.\label{eq:2.32}
\end{align}
In other words, such a function generates $0$-moments, namely,
\begin{align}
 \int_{-\infty}^{+\infty}\!\! x^n {g}(x)dx=0, \ n\geqslant 0.\label{eq:2.33}
 \end{align}
\end{definition}

\section{Functional Equation}

We are interested exclusively in the application of the characterization property  of the classical 2-OPS by means of the matrix  differential equation satisfied by the vector ${\bf U}$ to derive the systems involving the functionals $u_0$ and $u_1$. For this, we start with the following results.
\begin{lemma}{\cite{2}}
Let  $\{P_n\}_{n \geqslant 0}$ be a sequence of monic polynomials and let $\{u_n\}_{n \geqslant 0}$ be its  associated  dual sequence.
For any linear functional $u$ and integer $ m \geqslant 1$, the following statements are equivalent:
\begin{align*}
&{\rm (i)} \ \ \big< u \, ,\, P_{m-1} \big> \not = 0 \ ;\  \big< u \, ,\, P_{n} \big> = 0 ,\; n\geqslant m\ ;\\
&{\rm (ii)}\ \  \exists \ \lambda_\nu \in \mathbb{C} ,\ 0 \leqslant \nu \leqslant m-1 , \ \lambda_{m-1}\not= 0 , \
\mbox{ such that } u = \sum^{m-1}_{\nu = 0} \lambda_\nu {\it u}_\nu .\hskip 4.5cm
\end{align*}
 \end{lemma}
Now we can give the fundamental theorem.
\begin{theorem}
For the {\rm 2-OPS} $\{P_n\}_{n\ge 0}$ satisfying the third-order recurrence relation with the coefficients  given by $\eqref{eq:1.3}$-$\eqref{eq:1.5}$,
the associated vector functional ${\bf U} = {^t}(u_0 , u_1 )$ satisfies the following  matrix differential equation
\begin{align}
\big({\bf\Phi U}\big)' + {\bf\Psi U} = 0 ,\label{eq:3.1}
\end{align}
with ${\bf \Phi}$ and ${\bf \Psi}$ are two $2\times2$ polynomial matrices defined by
\begin{align}
{\bf \Phi}(x) = \begin{pmatrix}{\!1}&-\delta_0\\ \varphi(x) &{\, \tau_0}\end{pmatrix}\ ; \
{\bf\Psi}(x) = \begin{pmatrix}0&1\\ \psi(x) & \tau_1\end{pmatrix},\label{eq:3.2}
\end{align}
whose the elements are given by
$$\varphi(x) = -\frac{\delta_1}{\gamma}(x-\beta_0),\, \psi (x) = \frac{1}{\gamma}(x-\beta_0),\, \tau_0 = 1+\frac{\delta_1}{\gamma} \alpha_1,\,
\tau_1 = -\frac{\alpha_1}{\gamma}, \ \delta_0=s+r\ and \ \delta_1=s-r.$$
\end{theorem}
   \begin{proof}
   The proof of Equation \eqref{eq:3.1} is adapted from the characterization theorem \cite[Th. 3.1]{4} which in turn is based on an extension of the Hahn property \cite{16} to define the classical $d$-orthogonal polynomials. It only remains to determine the elements of the two matrices. To do this, we first proceed with
the action of the elements of the dual sequence $\{{v}_n\}_{n\ge 0}$ over the sequence $\{P_n\}_{n\ge 0}$. The action of both ${v}_0$ and ${v}_1$ over $P_n,\ n=0, 1, \ldots$ is particularly important which gives, respectively,
 \begin{align}
&\big< {v}_0\, ,\, P_0\big> = 1 \ ,\,\big< {v}_0\, ,\, P_1\big> = -\delta_0\ \; \hbox{and}\,\ \big< {v}_0\, ,\, P_n\big> = 0\,\; n\geqslant 2,\label{eq:3.3}\\
&\big< {v}_1\, ,\, P_0\big> = 0 \ ,\,\big< {v}_1\, ,\, P_1\big> = 1\ ,\,\big< {v}_1\, ,\, P_2\big> = -2\delta_1\ \; \mbox{and}\,\
\big< {v}_1\, ,\, P_n\big> = 0\ , \; n\geqslant 3.\label{eq:3.4}
\end{align}
By applying Lemma 3.1, taking into consideration \eqref{eq:3.3} and \eqref{eq:3.4}, one easily gets
 $$v_0 = u_0 -\delta_0 u_1\ \ \mbox{and}\ \ v_1 = u_1-2\delta_1u_2.$$
 Furthermore, setting $n=0$ in \eqref{eq:1.14} permits us to express  $u_2$ in terms of the couple $u_1$ and $u_0$ as
  \begin{align}
  \gamma_1u_2 = (x-\beta_0)u_0 -\alpha_1 u_1\quad (\mbox{with}\ \gamma_1=2\gamma).\label{eq:3.5}
  \end{align}
  In consequence, if we replace $u_2$ by its expression given in \eqref{eq:3.5}, one arrives at
  \begin{align*}
{v_0} &= u_0 -\delta_0 u_1,\\
{v_1} &= -\frac{\delta_1}{\gamma}(x-\beta_0) u_0 +\big(1+ \frac{\delta_1}{\gamma}\alpha_1\big)u_1,
 \end{align*}
or, equivalently, in the matrix form
\begin{align}
\begin{pmatrix}v_0\\ v_1
\end{pmatrix} =\begin{pmatrix} 1&-\delta_0\\ -\frac{\delta_1}{\gamma} (x-\beta_0)& 1+ \frac{\delta_1}{\gamma}\alpha_1\end{pmatrix}
\begin{pmatrix}u_0\\ u_1\end{pmatrix}.\label{eq:3.6}
\end{align}
From this it follows immediately that $\mathbf V=\mathbf\Phi\mathbf U$. If we set now $n=0$ and $n=1$ in  \eqref{eq:1.13}, we get $v_0'=-u_1$ and $v_1'=-2u_2$.
 After that, we use  \eqref{eq:3.5} to replace $u_2$ by its expression in the latter identity.\\
We differentiate  \eqref{eq:3.5} once, having regard to the foregoing, to straightly deduce that
\begin{align}
{\bf\Phi}(x)=\begin{pmatrix} 1&-\delta_0\\-\frac{\delta_1}{\gamma}(x-\beta_0)& 1+\frac{\delta_1}{\gamma}\alpha_1\end{pmatrix}\;
\hbox{and}\;\;
{\bf\Psi}(x)=\begin{pmatrix}0&1\\ \frac{1}{\gamma}(x-\beta_0)& -\frac{\alpha_1}{\gamma}\end{pmatrix}.\label{eq:3.7}
\end{align}
This clearly shows that the two matrices are nonsingular.
\end{proof}
\noindent As a direct consequence of this theorem we have the following corollary.
\begin{corollary}
Let $\eta:=2r\alpha_1 -\gamma$ and $\delta_0$, $\delta_1$  as in {\rm Theorem 3.1}. Then the functionals $u_0$ and $u_1$ necessarily satisfy one of the three systems.\\
{\rm 1.} For $\eta\ne0$ and $\delta_0\ne0$, we have
\begin{align}
 &[\phi(x)u_0]''+[\vartheta(x) u_0]'+ \chi(x) u_0 = 0,\label{eq:3.8} \\
 & \eta u_1=[\phi(x)u_0]' +\varrho(x)u_0.\label{eq:3.9}
  \end{align}
{\rm 2.} For $\eta\ne0$ and $\delta_0=0$, we have
\begin{align}
&\eta u_0''-\big[\vartheta(x)u_0\big]'-\chi(x)u_0=0,\label{eq:3.10}\\
&u_1=-u_0'.\label{eq:3.11}
\end{align}
{\rm 3.} If  $\eta=0$ and $\delta_0\ne0$, then
 \begin{align}
&\sigma(x)u_0' +\tau(x)u_0 =0,\label{eq:3.12}\\
&\delta_0 u_1'-u_1=u_0'.\label{eq:3.13}
\end{align}
The coefficients $\phi, \vartheta, \chi, \varrho, \sigma$ and  $\tau$ are polynomials of degree less than or equal to $1:$
\begin{align}
\phi(x)&=-\delta_1\delta_0(x-\beta_0)+\delta_1\alpha_1 +{\gamma}={\gamma}\det\big({\bf\Phi}(x)\big),\label{eq:3.14}\\
\vartheta(x)&=2s(x-\beta_0)-\alpha_1 ,\ \chi(x)=-x+\beta_0,\ \varrho(x)=\delta_0(x-\beta_0),\label{eq:3.15}\\
\sigma(x)&=\delta_1(x-\beta_0)-\alpha_1 ,\ \tau(x)=-x+\beta_0+\delta_1.\label{eq:3.16}
\end{align}
\end{corollary}
   \begin{proof}
According to Theorem 3.1, the pair of functionals $(u_0, u_1)$ satisfies Eq. \eqref{eq:3.1} with the two matrices ${\bf \Phi}(x)$ and ${\bf \Psi}(x)$ are
 given by \eqref{eq:3.7}. Starting from this equation, after some elementary computations based on the formulas given in Subsection 2.6, we readily obtain
\begin{align}
(\delta_1\alpha_1+{\gamma})u_1'-\alpha_1 u_1&=\delta_1[(x-\beta_0)u_0]'-(x-\beta_0)u_0,\label{eq:3.17}\\
\delta_0 u_1'-u_1&=u_0'.\label{eq:3.18}
\end{align}
The form of \eqref{eq:3.18}, which is the same as \eqref{eq:3.13}, clearly depends only on the parameter $\delta_0$. For this reason, we will restrict
 our attention to the conditions imposed on either of the values of $\eta$ and $\delta_0$. Observe  that the polynomial $\phi$ may be rewritten in the form $\phi(x)=-\delta_0\sigma(x)-\eta$ which clearly shows that $\eta$ and $\delta_0$ cannot both zero, since otherwise the matrix ${\bf \Phi}(x)$
lacks the non-singularity property.\\
For $\eta\ne0$ and $\delta_0\ne0$, if we multiply \eqref{eq:3.17} by $\delta_0$ and then combine with \eqref{eq:3.18}, we  get  \eqref{eq:3.9}.
Next, replace  $u_1$ by its expression given by \eqref{eq:3.9} in \eqref{eq:3.18} and then perform some elementary computations, we obtain that $u_0$ satisfies  \eqref{eq:3.8}.\\
 If now $\delta_0=0$, Eq. \eqref{eq:3.18} reduces to \eqref{eq:3.11} and \eqref{eq:3.17} becomes
\begin{align}
\eta u_1'+\alpha_1 u_1=2r[(x-\beta_0)u_0]'+(x-\beta_0)u_0.\label{eq:3.19}
\end{align}
On substituting  \eqref{eq:3.11} into \eqref{eq:3.19}, we easily find  \eqref{eq:3.10} with $\phi(x)=-\eta\ne0$.\\
 Finally,  for $\eta=0$ whenever $\delta_0\ne0$ and on account of \eqref{eq:3.18}, we easily check that the LHS of \eqref{eq:3.17} writes
\begin{align*}
(\delta_1\alpha_1+{\gamma})u_1'-\alpha_1 u_1=\alpha_1u_0',
\end{align*}
from which we readily deduce \eqref{eq:3.12}. To summarize, this allows us to conclude that the pair of functionals $(u_0, u_1)$ necessarily satisfies  one of the three systems  stated in the above assertions.
   \end{proof}

\noindent We close this section with a few remarks on the regularity   of the pair of functionals $u_0$ and $u_1$.
The vector ${\bf U}={^t}(u_0 , u_1 )$ is always assumed to be  regular in conformity with \cite[Def. 1.4]{1} which ensures the existence of the associated 2-OPS.\par

\noindent If we consider the two functionals $u_0$ and $u_1$ separately, the latter is never regular in the sense of the ordinary orthogonality since its first moment is
identically zero.
 Regarding  $u_0$, the question of its regularity cannot be answered in advance, except when it is $2$-symmetrical where we are in a position to assert that it is
 also non regular (see \cite[Rem. 1.2]{1}).
However, for $\eta=0$ and $\delta_0\ne0$, it is easy to see that the functional $u_0$  is   positive-definite.
 Indeed, Eq. \eqref{eq:3.12} shows that  $u_0$ coincides with the classical Hermite functional if  $\delta_1=0$, whilst for $\delta_1\ne0$ we find that $u_0$ coincides
  with the classical Laguerre functional.

\section{Integral Representations}

  The objective now is to seek integral representation for the functionals $u_0$ and $u_1$ in each of the special cases pointed out in \cite{1}. For this purpose,
  we shall assume that there exist two measures $\mu_0$ and $\mu_1$ and some contour ${\mathscr{C}}$ such that
\begin{align}
\big< u_0 \, ,\,f \big> &= \int_\mathscr{C}\!\! f(x) d\mu_0(x) \, ,\ \forall f\in{{\mathscr P}},\label{eq:4.1}\\
\big< u_1 \, ,\, f\big> &= \int_{\mathscr{C}}\!\! f(x)d\mu_1(x)\, ,\ \forall  f\in{{\mathscr P}}.\label{eq:4.2}
\end{align}
The contour of integration  is usually taken in the complex plane, but our goal here is to find integrals evaluated along (an interval of) the real axis.
Three configurations will  be  considered here: $\mathscr{C}= ]-\infty\ ; +\infty[ $, $ \mathscr{C}=[0\ ; +\infty[ $ or $ \mathscr{C}=]-\infty\ ; 0]$.\\
Note that in certain subcases, when ${\mathscr{C}}$ is taken to be the positive (or negative) real axis, addition of a Dirac mass at the origin would be necessary.\\
To illustrate our purpose, and to avoid repetition, we start by an exposition of the outline of  our procedure which we apply  to the first system \eqref{eq:3.8}-\eqref{eq:3.9}. This in fact generates the differential equations satisfied by the investigated weight functions, subject to certain boundary conditions.
The same reasoning is then used to process the other two systems, namely, \eqref{eq:3.10}-\eqref{eq:3.11} and \eqref{eq:3.12}-\eqref{eq:3.13} for which the main results will be successively announced without details. This method is similar to that  used in various works about  the integral representation of the functionals involved in the 2-orthogonality conditions (see for instance \cite{6,7,17,18,19}).\\
To this end,   we define the two polynomials $F$ and $G$ as follows
\begin{subequations} \begin{align}
 \hskip 1cm F(x)&=\phi(x) f''(x)-\vartheta(x) f'(x)+ \chi(x) f(x),\label{eq3a}\\
       G(x)&=-\phi(x) f'(x)+\varrho(x) f(x).\label{eq3b}\hskip 2cm
  \end{align}\end{subequations}
From \eqref{eq:3.8} and \eqref{eq:3.9}, using \eqref{eq3a}-\eqref{eq3b}, the action of both $u_0$ and $u_1$ over   $f$ gives
  \begin{align}
   0=\left<[\phi(x)u_0]''+[\vartheta(x) u_0]'+ \chi(x) u_0\, ,\, f\right>=\left<u_0\, ,\, F\right>,\label{eq:4.4}\\
   \eta\left<u_1\, ,\, f\right>=\left<u_0\, ,\, G\right>.\label{eq:4.5}
 \end{align}
 In what follows, ${g}_0$ and ${g}_1$  denote two functions representing the null functional as defined in \eqref{eq:2.32}-\eqref{eq:2.33};  $\zeta_0$ and $ \zeta_1$ stand for arbitrary constants (possibly zero). Since we have independent systems, we repeat use the same pairs of letters $({g}_0, {g}_1)$  and $(\zeta_0,\, \zeta_1)$ if necessary, when no confusion can arise.
 By the way, we also specify that it is assumed that in the two situations exposed below the weight functions $w_0$ and $w_1$ sought have the desired properties (at least twice differentiable).\par
 \medskip
 \noindent {\bf Situation 1.} There exist two weights $w_0$ and $w_1$ supported on  ${\mathscr{C}}=]-\infty\ ; +\infty[$, such that $d\mu_0(x)=w_0(x)dx$ and  $d\mu_1(x)=w_1(x)dx$. Thus \eqref{eq:4.1}-\eqref{eq:4.2} give rise to
  \begin{align}
\big< u_0 \, ,\,f \big> &= \int_\mathscr{C}\!\! f(x) w_0(x) dx  ,\ f\in{{\mathscr P}},\label{eq:4.6}\\
\big< u_1 \, ,\, f\big> &= \int_{\mathscr{C}} \!\! f(x) w_1(x)dx ,\ f\in{{\mathscr P}}.\label{eq:4.7}
\end{align}
{\bf 1a. } Assertion (1) of Corollary 3.1  states that  the system \eqref{eq:3.8}-\eqref{eq:3.9} occurs if we take $\eta\ne0$ and $\delta_0\ne0$.
Combining \eqref{eq:4.4} with \eqref{eq:4.6} we can   rewrite \eqref{eq:4.4} as
\begin{align*}
\int_\mathscr{C}\!\! F(x) w_0(x)dx =0, \ \forall f\in{{\mathscr P}}.
 \end{align*}
   We first perform integration by parts to get
 \begin{align*}
\begin{aligned}
\left[\phi w_0 f'- ((\phi w_0)' +\vartheta w_0)f\right]_{\mathscr{C}}+
\int_{\mathscr{C}}\!\!\left((\phi w_0)''+ (\vartheta w_0)'+ \chi w_0\right)f dx= 0.
\end{aligned}
\end{align*}
Hence, by imposing  the following two conditions
\begin{align}
\left[\phi w_0 f'-((\phi w_0)'+\vartheta w_0)f \right]_{\mathscr{C}}&=0, \, \forall f \in {{\mathscr P}},\label{eq:4.8}\\
\int_{\mathscr{C}}\!\!\left((\phi w_0)''+ (\vartheta w_0)'+ \chi w_0\right)f dx&= 0, \, \forall f \in {{\mathscr P}},\label{eq:4.9}
\end{align}
 we deduce from the latter equation  that
\begin{align}
\phi w_0''+(\vartheta+2\phi') w_0'+ (\chi+\vartheta') w_0 = \zeta_0 {g}_0.\label{eq:4.10}
\end{align}
 On the other hand,  application of \eqref{eq:4.7} and \eqref{eq:4.6} to \eqref{eq:4.5} readily shows that
  \begin{align}
 \eta w_1=\phi w_0' +(\varrho+\phi')w_0+\zeta_1{g}_1,\label{eq:4.11}
 \end{align}
  provided that $w_0$ satisfies also the additional boundary condition
  \begin{align}
\left[\phi w_0f\right]_{\mathscr{C}}=0 ,\; \forall f\in {{\mathscr P}}.\label{eq:4.12}
 \end{align}
\medskip
  \noindent {\bf 1b. } When $\eta\ne0$ and $\delta_0=0$, the same reasoning applies to  \eqref{eq:3.10}-\eqref{eq:3.11} yields
 \begin{align}
&\eta w_0''-\vartheta w_0'-(\chi+\vartheta')w_0=\zeta_0{g}_0,\label{eq:4.13}\\
&w_1=-w_0'+\zeta_1{g}_1,\label{eq:4.14}
\end{align}
under the conditions
 \begin{align}
&\left[\eta w_0 f'+(\vartheta w_0-\eta w_0')f\right]_{\mathscr{C}}=0 ,\; \forall f\in {{\mathscr P}},\label{eq:4.15}\\
& \left[w_0f\right]_{\mathscr{C}}=0 ,\; \forall f\in {{\mathscr P}}.\label{eq:4.16}
\end{align}
{\bf 1c. } Likewise, under the conditions on $\eta$ and $\delta_0$,  the system \eqref{eq:3.12}-\eqref{eq:3.13} leads to
  \begin{align}
&\sigma w_0'+\tau w_0=\zeta_0{g}_0,\label{eq:4.17}\\
&\delta_0 w_1'-w_1=w_0'+\zeta_1{g}_1,\label{eq:4.18}
 \end{align}
 satisfying the boundary conditions
  \begin{align}
 & \left[\sigma w_0f\right]_{\mathscr{C}}=0 ,\; \forall f\in {{\mathscr P}},\label{eq:4.19}\\
 &\left[(\delta_0w_1- w_0)f\right]_{\mathscr{C}}=0 ,\; \forall f\in {{\mathscr P}}.\label{eq:4.20}
 \end{align}\par
 \medskip
 \noindent {\bf Situation 2.} There exist two weights $w_0$ and $w_1$ supported on  ${\mathscr{C}}= [0\, ; +\infty[ $ $\big($or $]-\infty\, ; 0]\big)$ such that $d\mu_0(x)=\big(w_0(x)+\lambda_0\delta(x)\big)dx $ and $d\mu_1(x)=\big(w_1(x)+\lambda_1\delta(x)\big)dx$, where $\delta(x)$ stands for the standard  Dirac delta function and $\lambda_k $ ($k=0, 1$) are constants (possibly zero). This means that \eqref{eq:4.1}-\eqref{eq:4.2}  may be written as
\begin{align}
\big< u_0 \, ,\,f \big> &= \int_{\mathscr{C}}\!\!  f(x) d\mu_0(x)= \int_{\mathscr{C}}\!\! f(x)w_0(x) dx +\lambda_0 f(0)\, ,\,f\in{{\mathscr P}},\label{eq:4.21}\\
\big< u_1 \, ,\, f\big> &= \int_{\mathscr{C}} f(x) d\mu_1(x)= \int_{\mathscr{C}}\!\! f(x)w_1(x) dx +\lambda_1 f(0)\, ,\, f\in{{\mathscr P}}.\label{eq:4.22}
\end{align}
{\bf 2a.} As in Situation 1,  we start from \eqref{eq:3.8}-\eqref{eq:3.9} and combine \eqref{eq:4.4} with \eqref{eq:4.21} to write
\begin{align*}
\int_{\mathscr{C}} F(x) w_0(x)dx+\lambda_0 F(0) =0, \ \forall f\in{{\mathscr P}}.
 \end{align*}
Using integration by parts we obtain
\begin{align*}
\begin{aligned}
\left[\phi w_0 f'- ((\phi w_0)' +\vartheta w_0)f\right]_{\mathscr{C}}+\lambda_0 F(0)+
\int_{\mathscr{C}}\!\!\left((\phi w_0)''+ (\vartheta w_0)'+ \chi w_0\right)f dx= 0, \ \forall f\in{{\mathscr P}}.
\end{aligned}
\end{align*}
 If we impose the  two conditions
  \begin{align}
\left[\phi w_0 f'-((\phi w_0)'+\vartheta w_0)f \right]_{\mathscr{C}}+\lambda_0 F(0)&=0, \, \forall f \in {{\mathscr P}}, \label{eq:4.23}\\
\int_{\mathscr{C}}\!\!\left((\phi w_0)''+ (\vartheta w_0)'+ \chi w_0\right)f dx&= 0, \, \forall f \in {{\mathscr P}},\label{eq:4.24}
\end{align}
we deduce  from \eqref{eq:4.24} that $w_0$ satisfies Eq. \eqref{eq:4.10}. Further, application of \eqref{eq:4.22} and \eqref{eq:4.21} to \eqref{eq:4.5} gives \eqref{eq:4.11} provided that $w_0$ satisfies the additional condition
  \begin{align}
 \left[\phi w_0f\right]_{\mathscr{C}}+\eta\lambda_1f(0)-\lambda_0G(0)=0 ,\; \forall f\in {{\mathscr P}}.\label{eq:4.25}
 \end{align}
  {\bf 2b.}  Likewise, when $\eta\ne0$ with $\delta_0=0$,  if we start from \eqref{eq:3.10}-\eqref{eq:3.11}, we easily  check  that $w_0$ and $w_1$ satisfy the system \eqref{eq:4.13}-\eqref{eq:4.14}, under the  conditions
\begin{align}
 \left[\eta w_0 f'+(\vartheta w_0-\eta w_0')f\right]_{\mathscr{C}}-\lambda_0F(0)&=0 ,\; \forall f\in {{\mathscr P}},\label{eq:4.26}\\
  \left[w_0f\right]_{\mathscr{C}}+\lambda_0f'(0)-\lambda_1f(0)&=0 ,\; \forall f\in {{\mathscr P}}.\label{eq:4.27}
 \end{align}
{\bf 2c.} For $\eta= 0$ with $\delta_0\ne0$,  the system \eqref{eq:3.12}-\eqref{eq:3.13} readily gives \eqref{eq:4.17}-\eqref{eq:4.18}. This holds provided that the following  conditions being fulfilled
  \begin{align}
  \left[\sigma w_0f\right]_{\mathscr{C}}+\lambda_0\left[\sigma(0)f'(0)+\big(\sigma'(0)-\tau(0)\big)f(0)\right]=0 ,\; \forall f\in {{\mathscr P}},\label{eq:4.28}\\
 \left[(\delta_0w_1- w_0)f\right]_{\mathscr{C}}+(\delta_0\lambda_1-\lambda_0)f'(0)+\lambda_1 f(0)=0 ,\; \forall f\in {{\mathscr P}}.\label{eq:4.29}
 \end{align}
 Here and subsequently, we  constantly   take $\zeta_1=\zeta_0 = 0$, which is always possible to assume. In so doing, depending on the case and under the imposed conditions, we are in effect dealing with one of the following  systems.\\
  \vskip 0.05cm
   $({\cal S}_{1})\begin{cases}
 \phi w_0''+(\vartheta+2\phi') w_0'+ (\chi+\vartheta') w_0 = 0,\\
 \eta w_1=\phi w_0' +(\varrho+\phi')w_0.
 \end{cases} $\par
 \medskip
 $({\cal S}_{2})\begin{cases}
 \eta w_0''-\vartheta w_0'- (\chi+\vartheta') w_0 = 0,\\
  w_1=- w_0'.
 \end{cases}$\par
  \medskip
  $({\cal S}_{3})\begin{cases}
 \sigma w_0' +\tau w_0= 0,\\
 \delta_0 w_1'-w_1=w_0'.
 \end{cases}$\par
 \medskip

 \noindent  Since the polynomials $\phi, \vartheta$ and $\chi$ are of degree less than or equal to $1$, the homogeneous second order differential equations in both $({\cal S}_{1})$ and $({\cal S}_{2})$ are of type  Laplace which in certain cases  provide  solutions  given in terms of  special functions. On the other hand, since  $\deg\big(\sigma(x)\big)\leqslant1$ and $\deg\big(\tau(x)\big)=1$, we can assert  that the weight function $w_0$  in $({\cal S}_{3})$ is either  Laguerre  or  Hermite's one. \\
Note that using two different approaches, a variety of pair of orthogonality weights were given in various papers dedicated to this subject (see, e.g. \cite{6,7,8,9,17,18,19}).\\
 In this sense, an interesting example treated by Loureiro and Van Assche in \cite{17} deserves special attention. In fact, the authors obtained weight functions supported on a path where the role of the interval is replaced by a starlike set with three rays. More precisely, they gave weights via Airy function and its derivative supported on a starlike set ${\mathscr{S}}$ defined by ${\mathscr{S}}:={\mathscr{S}}_0\cup{\mathscr{S}}_1\cup{\mathscr{S}}_2$, with ${\mathscr{S}}_0$ corresponds to the positive real axis, ${\mathscr{S}}_1$ is the straight line starting  at $e^{2i{\pi}/3}\infty$ and ending at the origin, and ${\mathscr{S}}_2$ is the straight line starting  at $e^{4i{\pi}/3}\infty$ and ending at the origin. ${\mathscr{S}}_2$ and ${\mathscr{S}}_3$ are known in optics as Stokes' lines considered by himself to establish the reality of his phenomenon when he was studying the Airy integral (see \cite{20} and the references therein).
Except in this example (also encountered here), we are mainly concerned in this paper with orthogonality weights supported on (an interval of) the real axis. This often requires some restrictions on the free parameters.\\
Having disposed of this preliminary step, the task is now to find the suitable pair of functions $w_0$ and $w_1$ solution to the system involving  in each  special case.

\subsection{The special cases}

  To put into practice any system,  we first explicit the polynomial coefficients arising in the differential equations satisfied by
  $w_0$ and $w_1$.  As we have seen, the parameters $\eta$ and $\delta_0$  play an important role in the determination of the  system to examine.\par
    \medskip

 \noindent{\bf I.} For $\delta_1=\delta_0=0 \Leftrightarrow s=r=0$, which yields $\eta= -{\gamma}\ne0$, the polynomial coefficients are given by
 $\ \phi(x)={\gamma},\, \vartheta(x)=-\alpha_1 ,\, \chi(x)=-(x-\beta_0),\,\varrho(x)=0$.\\
Under the  two conditions $\eta\ne0$ and $\delta_0=0$, Assertion (2) of Corollary 3.3 shows that the appropriate system to consider is therefore $({\cal S}_{2})$.\par
\noindent   Accordingly, we have to solve the following equations
\begin{align}
&{\gamma}w_0''- \alpha_1 w_0'- (x-\beta_0) w_0 = 0,\label{eq:4.30}\\
 &w_1 = -w_0',\label{eq:4.31}
\end{align}
under the boundary conditions
\begin{subequations}\begin{align}
&\left[\gamma w_0 f'+(\alpha_1 w_0-\gamma w_0')f\right]_{\mathscr{C}}=0 ,\; \forall f\in {{\mathscr P}},\label{eq32a}\\
& \left[w_0f\right]_{\mathscr{C}}=0 ,\; \forall f\in {{\mathscr P}}.\label{eq32b}
\end{align}\end{subequations}
By choosing specific values of the parameters $\gamma$, $ \alpha_1$ and $\beta_0$, three subcases deserve to be discussed here.
It goes without saying that Eq. \eqref{eq:4.31} remains  unchanged. \\
 \medskip
\noindent {\bf - Subcase I.1}  Fix $\gamma=1$ and set $\alpha_1:=2\alpha >0$, $\beta_0 =0$, the first equation writes
\begin{align}
w_0''- 2\alpha w_0'- x w_0 = 0. \label{eq:4.33}
\end{align}
\begin{proposition}
  The two functionals $u_0$ and $u_1$ admit the integral representations \eqref{eq:4.6}-\eqref{eq:4.7}, with
  \begin{align*}
w_0(x)=e^{\frac{2}{3}\alpha^3}e^{\alpha x}{\rm Ai}(x+\alpha^2)\quad \mbox{and}\quad
w_1(x)=-e^{\frac{2}{3}\alpha^3}\left(e^{\alpha x}{\rm Ai}(x+\alpha^2)\right)'\, ,\; \alpha>0\ ,\; x\in\mathbb{R},
\end{align*}
subject to \eqref{eq32a}-\eqref{eq32b}, where ${\rm Ai}$ is the Airy function of the first kind.
\end{proposition}
   \begin{proof}
We first perform the transformation $w_0(x)=e^{\alpha x}W(x)$ in \eqref{eq:4.33}, to deduce that the function $W$ satisfies the Airy-type differential
 equation \eqref{eq:2.12} whose the general solution  is
\begin{align*}
W(x)= c_1{\rm Ai}(x+\alpha^2)+c_2{\rm Bi}(x+\alpha^2),\ \ x\in \mathbb{R},
\end{align*}
where $c_1$ and $c_2$ are arbitrary constants. Since ${\rm Ai}(x)$ decays and ${\rm Bi}(x)$ grows  as $x\to+\infty$, taking $c_2=0$, a solution of Eq.
\eqref{eq:4.33} satisfying \eqref{eq32a}-\eqref{eq32b} can simply be written as
\begin{align*}
w_0(x)=c_1e^{\alpha x}{\rm Ai}(x+\alpha^2), \ \ x\in \mathbb{R}.
\end{align*}
The constant $c_1$  is determined  by means  of the normalization condition, that is,
\begin{align*}
1=\left<u_0\, ,\, 1 \right>=\int_{-\infty}^{+\infty}\!\! w_0(x) dx= c_1 \int_{-\infty}^{+\infty}\!\!e^{\alpha x}{\rm Ai}\left(x+\alpha^2\right)dx.
\end{align*}
Thanks to Laplace's transform of the Airy function  \cite[Sec. 9]{5}, we find that the value for the integral in the right is given by
\begin{align*}
\int_{-\infty}^{+\infty}\!\!e^{\alpha x}{\rm Ai}\left(x+\alpha^2\right)dx=e^{-\alpha^3}\int_{-\infty}^{+\infty}\!\!e^{\alpha t}{\rm Ai}(t)dt
 = e^{-\alpha^3}e^{\frac{1}{3}\alpha^3}=e^{-\frac{2}{3}\alpha^3}.
\end{align*}
It turns out that $c_1=e^{\frac{2}{3}\alpha^3}$, and so the weight function $w_0$ writes
\begin{align}
w_0(x)= e^{\alpha (x+\frac{2}{3}\alpha^2)}{\rm Ai}(x+\alpha^2)\, ,\; \alpha>0\ ,\; x\in\mathbb{R}.\label{eq:4.34}
\end{align}
From \eqref{eq:4.31} it is straightforward that the second weight function  is given by
\begin{align}
w_1(x)=-e^{\frac{2}{3}\alpha^3}\Big(e^{\alpha x}{\rm Ai}(x+\alpha^2)\Big)'\, ,\; \alpha>0\ ,\; x\in\mathbb{R}.\label{eq:4.35}
\end{align}
Now, we must check that the  two vanishing conditions \eqref{eq32a}-\eqref{eq32b} are meet at the endpoints of the real line. \\
As $x\to+\infty$, use of $\eqref{eq:2.13}$  successively gives
\begin{align*}
e^{\alpha x}\!{\rm Ai}(x+\alpha^2)\sim {\frac{e^{-\frac{2}{3}{(x+\alpha^2)}^{\frac{3}{2}}+\alpha x}}{2\sqrt{\pi}{(x+\alpha^2)}^\frac{1}{4}}}\to 0\ \hbox{and}\
e^{\alpha x}\!{\rm Ai}'({x+\alpha^2})\sim -{\frac{e^{-\frac{2}{3}{(x+\alpha^2)}^{\frac{3}{2}}+\alpha x}}{2\sqrt{\pi}{(x+\alpha^2)}^{-\frac{1}{4}}}}\to 0.
\end{align*}
 On the other hand,  for large negative argument,  $\eqref{eq:2.14}$ readily gives
 \begin{align*}
e^{\alpha x}{\rm Ai}\left(x+\alpha^2\right)\to 0\quad\hbox{and}\quad e^{\alpha x}{\rm Ai}'\left(x+\alpha^2\right)\to 0, \quad \mbox{as}\ x\to-\infty.
\end{align*}
In the light of these last results, 
a trivial verification shows that the boundary conditions \eqref{eq32a}-\eqref{eq32b} are fulfilled which completes the proof.
   \end{proof}
\noindent {\bf Remarks 4.1.}
{\bf (a)} In \eqref{eq:4.34}, if we change $x\to-x$ and set $\alpha=t$, we recover a solution of the heat equation $\partial_tu(x,t)=\partial_{xx}u(x,t)$ expressed in terms of Airy function as
$$u(x,t)= e^{\frac{2}{3}t^3 -xt}{\rm Ai}(t^2-x).$$
This function appears in different domains of sciences. For more details see \cite{10}.\\
{\bf (b)} Using \eqref{eq:2.4}, we find that $w_0$ can be expressed in terms of  Macdonald's function in the form
\begin{align*}
w_0(x)=\frac{e^{\frac{2}{3}\alpha^3}}{\pi{\sqrt{3}}}e^{\alpha x}{(x+\alpha^2)^\frac{1}{2}}K_\frac{1}{3}\Big(\frac{2}{3}(x+\alpha^2)^\frac{3}{2}\Big),\; \alpha>0 ,\;
x\in\mathbb{R}.
\end{align*}
{\bf (c)} Likewise, due to \eqref{eq:2.8}, $w_0$ can simply be written in terms of the Gauss-Airy function as follows
\begin{align*}
w_0(x)={\rm GAi}(x , \alpha , 1) ,\; \alpha>0 ,\; x\in\mathbb{R}.
\end{align*}
Based on this last identity, taking into account \eqref{eq:2.9}, we deduce that the moments of $w_0(x)$ can be expressed as particular values  of the three-variable Hermite polynomials, that is,
 \begin{align*}
\int_{-\infty}^{+\infty}\! t^nw_0(t)dt=\int_{-\infty}^{+\infty}\! t^n{\rm GAi}(t,\alpha,1)dt={_3}H_n(0,\alpha,1) ,\ \ n\geqslant 0.
\end{align*}
\noindent {\bf - Subcase I.2}  For $\gamma=1$, $\beta_0=0$ and $\alpha_1=0$ in \eqref{eq:4.30}, we see that  $w_0$ satisfies  Airy's Equation
\begin{align}
 w_0''- x w_0 = 0.\label{eq:4.36}
\end{align}
  This is precisely the subcase treated in \cite[Prop. 3.9]{17} where the authors gave  two weight functions  supported on  the starlike set ${\mathscr{S}}$.
   This pair of weights is simply Airy function and its first derivative (see the above reference for further details)
  \begin{align*}
  w_0(x)={\rm Ai}(x)\quad \hbox{and}\quad w_1(x)=-{\rm Ai}'(x).
  \end{align*}
It is also shown that the vanishing conditions are fully met at the endpoints of ${\mathscr{S}}$.\par
\noindent Let  the sequence ${(u_k)}_n;  n=0, 1, 2, \ldots $   stands for the moments  of  $u_k, \;k\in \{0, 1\}$.\\
 On applying the results given \cite[Corollary 3.8]{21} stating the moments of the $d$-symmetric functionals, adapted  for $d=2$,  we obtain
\begin{align*}
  {(u_0)}_{3n}&=\frac{(3n)!}{3^n(n!)} \qquad\mbox{and}\quad {(u_0)}_{3n+1}={(u_0)}_{3n+2}=0, \ n\geqslant 0,\hskip 2cm\\
 \hskip 2cm {(u_1)}_{3n+1}&=\frac{(3n+1)!}{3^n(n!)} \quad\mbox{and}\quad {(u_1)}_{3n}={(u_1)}_{3n+2}=0, \ n\geqslant 0.
\end{align*}
Similar results have been also given in \cite{17}.\\
\noindent {\bf- Subcase I.3}  Likewise, if we set $\gamma={1/9}$ and take
 $\beta_0=0$, $\alpha_1=0$ in \eqref{eq:4.30}, we readily see that $w_0$ satisfies
 \begin{align}
 w_0''- 9x w_0 = 0.\label{eq:4.37}
 \end{align}
 This last equation has been considered in the complex $z$-plane by Stokes in his first work on the so-called Stokes phenomenon problem published in \cite{22}.
 He gave two independent  power series solutions of this equation, which for small $|z|$ could be used in the computation of the general solution.\\
Similarly to {Subcase I.2}, due to {\eqref{eq:2.7}}, we conclude that
 \begin{align*}
w_0(x)= 3^{\frac{2}{3}}{\rm Ai}(3^{\frac{2}{3}}x)\quad \hbox{and}\quad w_1(x)= -3^{\frac{4}{3}}{\rm Ai}'(3^{\frac{2}{3}}x).
\end{align*}
The 2-OPS $\{P_n\}_{n\geqslant0}$ associated with  these pair of weights are  also 2-symmetric. They are  related to  the hypergeometric polynomials
${_1}F_2\big(\!-n\, ; 1+\alpha\, ,1+\beta\, ; x\!\big), n\geqslant0$, $\alpha, \beta>-1$, studied in detail in \cite{7} and denoted therein $B_n^{\alpha,\beta}(x)$.
 To be more precise,  let ${\hat B}_n^{\alpha,\beta}(x)$ denotes the monic polynomial corresponding to $B_n^{\alpha,\beta}(x)$. This enables us to write
 ${\hat B}_n^{\alpha,\beta}(x)=(-1)^n(1+\alpha)_n(1+\beta)_n B_n^{\alpha,\beta}(x)$.\\
 Application of the cubic decomposition of the $2$-symmetric sequence $\{P_n\}_{n\geqslant0}$ gives
\begin{align*}
P_{3n}(x)={\hat B}^{{-\frac{2}{3}},{-\frac{1}{3}}}_{n}(x^3)\;  ;\;P_{3n+1}(x)=x{\hat B}^{{-\frac{1}{3}},\frac{1}{3}}_{n}(x^3)\;  ;\;
P_{3n+2}(x)=x^2{\hat B}^{\frac{1}{3},\frac{2}{3}}_{n}(x^3).
\end{align*}
It is also shown  that   $\{{\hat B}_n^{\alpha,\beta}\}_{n\geqslant0}$  is a classical 2-OPS associated with the weights
$w^{\alpha,\beta}_0(x)=\lambda_1 x^\beta\rho_\nu(x)$ and $w^{\alpha,\beta}_1(x)=-\lambda_2\left(x^{\beta+1}\rho_\nu(x)\right)'$,  $x\geqslant0$,
$\alpha\geqslant\beta >-1$, $\nu=\alpha-\beta $, and $\lambda_k:=\lambda_k(\alpha,\beta)=1/{\Gamma(k+\alpha)\Gamma(k+\beta)},\ k\in \{1, 2\}$, where $\rho_\nu(x)$
is the scaled Macdonald function defined by \eqref{eq:2.2}.
 Mention finally that similar results have been obtained independently in \cite{9}, where the authors considered the orthogonality conditions  w.r.t. the system of two positive weight functions $\left(x^\beta\rho_\nu(x) , x^\beta\rho_{\nu+1}(x)\right)$,  with $x\geqslant0,\, \nu\geqslant 0$ and $\beta>-1$.\par
 \medskip

 \noindent{\bf II.} For $\delta_1=\delta_0\ne0 \Leftrightarrow s\ne0, r=0$, with $\eta= -{\gamma}\ne0$,  we have \\
\indent $\phi(x)=-s^2(x-\beta_0)+s\alpha_1+\gamma,\, \vartheta(x)=2s(x-\beta_0)-\alpha_1$,\\
\indent $\chi(x)=-(x-\beta_0),\, \varrho(x)=s(x-\beta_0)$.\par
\noindent Since we are working under the two conditions $\eta\ne0$ and $\delta_0\ne0$, the first system $({\cal S}_{1})$ is the one to be considered which we may write
\begin{align*}
&\left(-s^2(x- \beta_0)+s \alpha_1 +\gamma\right) w_0''+ \left(2s(x-\beta_0-s)-\alpha_1\right) w_0'+\left(-x+\beta_0+2s\right) w_0 = 0 ,\\
&\gamma w_1 = \left(s^2(x- \beta_0)-s \alpha_1 -\gamma\right)w_0'+s\left(-x+\beta_0+2s\right)w_0.
\end{align*}
Under the assumptions that $\gamma=1$ and $s=-1$, by  moving the singularity of the second order differential equation at the origin, we see that $\alpha_1=\beta_0+1$.\\
 We thus  recover the classical 2-OPS treated  in \cite{6} (referred to there  as Case $A$).\\
 By setting $\beta_0=\alpha+2$ and so $\alpha_1=\alpha+3$, for fixed $\alpha$ (taken as a free parameter),  a family of $2$-OPS has been highlighted giving rise to 2-OPS  analogous to the  classical Laguerre ones admitting the integral representations \eqref{eq:4.21}-\eqref{eq:4.22} with $w_0(x,\alpha)=\omega_{\alpha}(x)e^{-x-1}$ and $w_1(x,\alpha)=-\left(\omega_{\alpha+1}(x)e^{-x-1}\right)'$,  $ \alpha>-1 , \, {\mathscr{C}}= [0 ; +\infty[$ and $ \lambda_1=\lambda_0=0$, where $\omega_{\alpha}(x)$ is the scaled modified Bessel function of the first kind given by \eqref{eq:2.1}. Adopting the same approach as in \cite{9}, similar results have been obtained independently in \cite{8} with the two positive weight functions $\left(e^{-cx}\omega_{\alpha}(x),e^{-cx}\omega_{\alpha+1}(x)\right)$, $x\geqslant0 ,\;\alpha>-1$ and $c>0$.\par
  \medskip

\noindent{\bf III.} For $\delta_1=-\delta_0\ne0 \Leftrightarrow s=0, r\ne0$,  with  $\eta= 2r\alpha_1-{\gamma}$, the involving polynomial coefficients in this case are in turn given by\\
\indent $\, \phi(x)=r^2(x-\beta_0)-r\alpha_1+\gamma, \ \vartheta(x)=-\alpha_1,\ \chi(x)=-(x-\beta_0), \ \varrho(x)=r(x-\beta_0),\,$ \\
\indent  $\, \sigma(x)=-r(x-\beta_0)-\alpha_1$ and $ \tau(x)=-x+\beta_0-r$.\\
Depending on the parameter $\eta$, we have in fact to deal with two different systems.
The first one   is in fact of   form $({\cal S}_{1})$ arising when $\eta\ne0$, since we also have $\delta_0\ne0$.
For $\eta=0$, however,  we will treat a system of type $({\cal S}_{3})$.\\
\noindent {\bf- Subcase III.1} For $\eta \neq 0$, we have  the following system
  \begin{align*}
   &\left(r^2(x-\beta_0)-r\alpha_1+\gamma\right) w_0''+(2r^2-\alpha_1) w_0'-(x-\beta_0) w_0 = 0 , \\
   &(2r\alpha_1-\gamma) w_1 = \left(r^2(x-\beta_0)-r\alpha_1+\gamma\right)w_0'+r\left(x-\beta_0+r\right)w_0.
  \end{align*}
 By setting $r^2=1$, and then making a linear transformation to position the singularity of the second order differential equation at the origin, we get $\beta_0+r\alpha_1 - \gamma=0$. The corresponding family of classical 2-OPS  has already been mentioned  in \cite{6} (denominated therein Case $B$) without addressing the issue of the integral representations for its orthogonality measures. That is what we will do right after.\\
Fix $r=-1$, so that $\gamma=\beta_0-\alpha_1$ and $\eta=-\beta_0-\alpha_1$. Then  $w_0$ and $w_1$  satisfy
   \begin{align}
& x w_0''+(2-\alpha_1) w_0'+(\beta_0-x) w_0 = 0 ,\label{eq:4.38}\\
&(\alpha_1+\beta_0)  w_1 = -xw_0'+\left(x-\beta_0-1\right)w_0,\label{eq:4.39}
\end{align}
and subject to
\begin{subequations}
\begin{align}
\left[x w_0 f'+\left((\alpha_1-1) w_0 -x w_0'\right)f \right]_{\mathscr{C}}&=0, \label{eq40a}\\
\left[x w_0 f\right]_{\mathscr{C}}&= 0. \label{eq40b}
\end{align}
\end{subequations}
The two parameters $\alpha_1$ and $\beta_0$ must satisfy $\alpha_1\ne\pm\beta_0$, since neither $\gamma$ nor $\eta$ is zero.\par
\noindent Let $ p=-\frac{1}{2}(\alpha_1+\beta_0)$ and $q=-\frac{1}{2}(\alpha_1-\beta_0)$, so that  instead of $\alpha_1,\beta_0$, we shall often write  $p$ and $q$ in various formulas.  More precisely,  we need the expressions
$$1\pm p=1\mp\frac{1}{2}(\alpha_1+\beta_0),\, 1\pm q=1\mp\frac{1}{2}(\alpha_1-\beta_0),\, \ p-q=-\beta_0 \ \mbox{ and } \ p+q=-\alpha_1.$$
 To write simplified formulas,  we also introduce  two modified Tricomi functions that are used here and later in Case VI.
  For $c>0$,  define
 \begin{subequations}\begin{align}
 \hskip 1cm {\mathscr{U}(cx;p,q)} &:=e^{-cx}U\big(1+p\, ;2+ p+q\, ; 2cx\big),&&\mbox{if}\ x\geqslant0,\label{eq41a}\\
       {\mathscr{V}(cx;q,p)} &:=e^{cx}U\big(1+q\, ;2+ p+q\, ; -2cx\big), &&\mbox{if}\ x\leqslant0.\label{eq41b}\hskip 1cm
  \end{align}\end{subequations}
  Observe that $p$ and $q$ are not interchangeable in both $\mathscr{U}$ and $\mathscr{V}$ because of the first parameter in $U$.
 By working out the derivatives, using \eqref{eq:2.19}, we get
    \begin{subequations} \begin{align}
 \hskip 1cm \frac{d}{dx}{\mathscr{U}(cx;p,q)} &=c{\mathscr{U}(cx;p,q)}-2c{\mathscr{U}(cx;p,1+q)},&&\mbox{if}\ x\geqslant0,\label{eq42a}\\
      \frac{d}{dx} {\mathscr{V}(cx;q,p)} &=-c{\mathscr{V}(cx;q,p)}+2c{\mathscr{V}(cx;q,1+p)}, &&\mbox{if}\ x\leqslant0.\label{eq42b}\hskip 0.5cm
  \end{align} \end{subequations}
We are now ready to state the following result.
 \begin{proposition}
 Under the assumptions $-1<p<0$ and $-1<q<0$ with $-2<p+q<-1$, the functionals $u_0$ and $u_1$ admit the integral representations \eqref{eq:4.6}-\eqref{eq:4.7} over ${\mathscr{C}}=\mathbb{R}$, with
 \begin{align}
w_0(x)&=\left\{\begin{aligned}
&k_1 \mathscr{U}(x;p,q), &&  \mbox{if}\ x\geqslant0,\\
&k_2\mathscr{V}(x;q,p), && \mbox{if}\ x\leqslant0,
\end{aligned}\right.\label{eq:4.43}\\
w_1(x)&=\left\{\begin{aligned}
&{\tilde{k}_1}\big[(1+\beta_0)\mathscr{U}(x;p,q)-2x\mathscr{U}(x;p,1+q)\big], && \mbox{if}\ x\geqslant0, \\
&\tilde{k}_2\big[(1+\beta_0- 2x )\mathscr{V}(x;q,p)+2x\mathscr{V}(x;q,1+p)\big], &&\mbox{if}\ x\leqslant0.
\end{aligned}\right.\label{eq:4.44}
\end{align}
The constants $k_1$, $k_2$, $\tilde{k}_1$ and $\tilde{k}_2$ are given by
   \begin{align}
   k_1=\frac{p\Gamma(1-q)}{\Delta(p,q)}, \, k_2=\frac{q\Gamma(1-p)}{\Delta(p,q)}, \, \tilde{k}_1=\frac{q\Gamma(-q)}{2\Delta(p,q)}\ \mbox{and}\
   \tilde{k}_2=\frac{q\Gamma(-p)}{2\Delta(p,q)}, \label{eq:4.45}
   \end{align}
    with
$$\Delta(p,q):={\Gamma(-p-q)}\Big[p\,{_2}F_1\big(1\, , 1+p\, ; 1-q\, ; -1\big)+q\,{_2}F_1\big(1\, , 1+q\, ; 1-p\, ; -1\big)\Big],$$
   where ${_2}F_1$ is the Gauss hypergeometric function.
\end{proposition}
\begin{proof}
  We start with the observation that the second-order differential equation \eqref{eq:4.38} is of Laplace type. It is well known that such equations have
  Laplace integrals as solutions. We are thus looking for a function $w_0$ of the form
 \begin{align}
 w_0(x) = \int_{\cal C}e^{xt}v(t)dt,\label{eq:4.46}
\end{align}
where the function $v(t)$ and the path of integration ${\cal C}$ have to be chosen so that \eqref{eq:4.46} becomes a solution of \eqref{eq:4.38}.
See, e.g. \cite[Sec.VIII]{Ince} for details.\\
Let $P$ and $Q$ be polynomials such that $ P(t)=(2-\alpha_1)t+\beta_0 $ and $ Q(t)=t^2-1$.
 Assuming the right to apply the differentiation operation  under the sign of integration, we may compute  the derivatives $w_0'(x)$ and  $w_0''(x)$.
  Substituting into \eqref{eq:4.38} and  integrating by parts yields
 \begin{align}
\left[e^{xt}Q(t)v(t)\right]_{\cal C}-\int_{\cal C}e^{xt}\Big\{\frac{\partial}{\partial t}\big[Q(t)v(t)\big]-P(t)v(t)\Big\}dt=0,\label{eq:4.47}
\end{align}
Obviously,  Eq.  \eqref{eq:4.47}  still satisfied when
\begin{align}
\left[e^{xt}Q(t)v(t)\right]_{\cal C}=0 \quad \mbox{and}\quad \frac{\partial}{\partial t}\big[Q(t)v(t)\big]=P(t)v(t).
\label{eq:4.48}
\end{align}
A trivial verification shows that the differential equation in \eqref{eq:4.48} holds when the function $v(t)=(t+1)^{p}(t-1)^{q}$ (up to a multiplicative constant), and so  the boundary conditions  gives rise to the equation $(t+1)^{p+1}(t-1)^{q+1}e^{xt}=0$.
   Observe that the first two roots of this equation are $t=-1$ for $\, p+1>0$ and $t=1$ for $\, q+1>0$.\break
  If now $x$ be restricted to positive values, a third root is given by $t=-\infty$, while when $x$ is negative it is given by $t=+\infty$. \\
This gives a variety of paths, each  leads to a particular solution. Possible paths such that the integrand vanishes identically at their endpoints are
\begin{align*}
&{\rm 1.} \ -\infty<t<-1, \quad x>0, \ p>-1.\\
&{\rm 2.} \ +1<t<+\infty, \quad x<0, \ q>-1.\\
&{\rm 3.}\ -1<t<+1, \quad \,  x\in\mathbb{R},\ p>-1, \, q>-1.\hskip 5cm
 \end{align*}
Replacing $v(t)$ by its explicit expression in \eqref{eq:4.46}, taking into consideration these paths, three particular solutions to the equation \eqref{eq:4.38} can be represented as follows
\begin{align*}
{\rm 1.} \  w^1_0(x)&=c_1\int_{-\infty}^{-1}e^{xt}(t+1)^{p}(t-1)^{q}dt,\quad \mbox{if}\ x>0,\ p>-1,\\
{\rm 2.} \  w^2_0(x)&=c_2\int_1^{+\infty}e^{xt}(t+1)^{p}(t-1)^{q}dt,\quad \mbox{if}\ x<0, \ q>-1,\\
{\rm 3.} \  w^3_0(x)&=c_3\int_{-1}^1e^{xt}(t+1)^{p}(t-1)^{q}dt, \quad \mbox{for}\ x\in \mathbb{R}, \ p>-1, \ q>-1,\hskip 1cm
\end{align*}
 where $c_1$, $c_2$, and $c_3$ are  constants. With appropriate changes of the variable in these integrals, due to \eqref{eq:2.20} and \eqref{eq:2.21},
 these  solutions can be written  as follows
\begin{align*}
   w^1_0(x)&=k_1e^{-x}U\big(1+p\, ; 2+p+q\, ; 2x\big),  \quad \mbox{if}\ x>0,\  p>-1,\\
   w^2_0(x)&=k_2e^{x}U\big(1+q\, ;2+p+q\, ; -2x\big), \quad\mbox{if}\ x<0,\  q>-1,\\
   w^3_0(x)&=k_3e^{-x} M\big(1+p\, ; 2+p+q\, ; 2x\big), \quad\mbox{for}\ x\in \mathbb{R},\  p>-1,\, q>-1,\hskip 1cm
\end{align*}
with $k_1$, $k_2$ and $k_3$ are new constants. \\
Before pursuing our reasoning, several observations are worth particular mention.\\
\noindent{\bf (a)} By virtue of \eqref{eq:2.23}-\eqref{eq:2.24}, the function $w^3_0$ grows exponentially when $x\to \pm\infty$.\\
\noindent{\bf (b)} Under the additional  condition $1+p+q<0$ (i.e., $\alpha_1>1$), due to the results listed in Table 1 Sec. 2, we have
   \begin{align*}
 \underset{x\to 0^+}{\lim}w_0^1(x)=\frac{\Gamma(\alpha_1-1)}{\Gamma(-q)}k_1 \quad\mbox{and}\quad
\underset{x\to 0^-}{\lim}w_0^2(x)=\frac{\Gamma(\alpha_1-1)}{\Gamma(-p)}k_2,\quad -1<p, q<0.
\end{align*}
From now on, Tables 1 and 2   will be referred to without reference to Section 2.\\
\noindent{\bf (c)} We easily check that  $w^1_0(x)\to 0$ as $x\to +\infty$ and $w^2_0(x)\to0$ as $x\to -\infty$. \\
\noindent{\bf (d)} Combining the imposed conditions on $p$ and $q$ gives   $-2<p+q<-1$ which, in turn, provides  the inequalities
 $1<\alpha_1<2$ and $\alpha_1-2<\beta_0<2-\alpha_1$.\par

 \noindent{\bf (e)} Two  limiting cases naturally arise here, namely, $1+p+q=0$ and $2+p+q=0$. They will be  examined aside.\par
 \smallskip
 \noindent From what has been shown, we may conclude that the weight function $w_0$ satisfying \eqref{eq40a}-\eqref{eq40b}
 can be composed of both functions $\mathscr{U}$ and  $\mathscr{V}$ supported  on two touching intervals as stated in \eqref{eq:4.43}.\\
We next turn to computing the constants $k_1$, $k_2$.  To do this, for $1+p+q<0$, we deduce from the  observation {\bf (b)} that both limits
 $\underset{x\to 0^+}{\lim}\mathscr{U}(x;p,q)$ and $\underset{x\to 0^-}{\lim}\mathscr{V}(x;q,p)$ are finite so that the continuity of $w_0(x)$ at $x=0$  is plainly ensured if we take $\underset{x\to 0^+}{\lim}w_0(x)=\underset{x\to 0^-}{\lim}w_0(x)$. This leads to the equation
\begin{align}
\Gamma(-p)k_1-\Gamma(-q)k_2=0.\label{eq:4.49}
\end{align}
 Another equation involving $k_1$ and  $k_2$  is derived from  $\left< u_0\, ,\, 1\right>=1$. This gives
\begin{align}
I_1({p,q}) k_1 +I_2({p,q}) k_2=1, \label{eq:4.50}
\end{align}
where $I_1({p,q})$ and $I_2({p,q})$ stand for the  two integrals
 $$I_1({p,q})\!:=\!\!\int_0^{+\infty}\!\!\!\!\!e^{-x} U\big(1+p  ; 2+p+q ; 2x\big)dx\, ;\,
I_2({p,q})\!:=\!\!\int_0^{+\infty}\!\!\!\!\!e^{-x} U\big(1+q  ; 2+p+q ; 2x\big)dx.$$
An application of  \eqref{eq:2.22} allows us to get
$$I_1({p,q})\!=\!\frac{\Gamma(\!-p\!-\!q)}{\Gamma(1-q)}{_2}F_1\big(1 , 1+p ; 1-q ; -1\big)\, ; \,
I_2({p,q})\!=\!\frac{\Gamma(\!-p\!-\!q)}{\Gamma(1-p)}{_2}F_1\big(1 , 1+q ; 1-p ; -1\big).$$
The solution of   \eqref{eq:4.49}-\eqref{eq:4.50} is then the pair of constants ${k}_1$ and ${k}_2$ stated  in \eqref{eq:4.45}.\\
 The proof is completed by showing that the vanishing conditions \eqref{eq40a}-\eqref{eq40b} are fulfilled. Observe  first
 that the  function $w_0(x)$ given by \eqref{eq:4.43} tends to zero as $x\to\pm\infty$, and so does $xw_0(x)$.
 We now differentiate   $w_0$  to get
 \begin{align*}
xw_0'(x)=\left\{\begin{aligned}
&k_1x\big[\mathscr{U}(x;p,q)-2\mathscr{U}(x;p,1+q)\big], &&\mbox{if}\ x\geqslant0,\\
&k_2 x\big[-\!\!\mathscr{V}(x;q,p)+2\mathscr{V}(x;q,1+p)\big], &&\mbox{if}\ x\leqslant0.
\end{aligned}\right.\end{align*}
Due to Tables 1\&2, we find that  $\underset{x\to 0^+}{\lim}xw_0'(x)=\underset{x\to 0^-}{\lim}xw_0'(x)=0$. Moreover,
we see that $xw’_0(x)\to 0$ as $x\to\pm\infty$. We conclude that  the required conditions are met.
Our next step is to derive the second function $w_1$. For this purpose, if we substitute $xw’_0(x)$ and $w_0(x)$ by their expressions
into \eqref{eq:4.39},  we immediately obtain \eqref{eq:4.44}.
Based again on \eqref{eq:4.49} we conclude similarly that the continuity of $w_1$ at $x=0$ occurs because the equality
$\underset{x\to 0^+}{\lim}w_1(x)=\underset{x\to 0^-}{\lim}w_1(x)$ holds.
\end{proof}\par
\noindent {\bf Remark 4.2.}
 If we seek integral representations via weights supported only on the positive real axis,  we need consider  measures with add of Dirac masses at $x=0$,
  that is, $d\mu_0(x)=\big(w_0(x)+\lambda_0\delta(x)\big)dx\,$ and $\,d\mu_1(x)=\big(w_1(x)+\lambda_1\delta(x)\big)dx$ with
\begin{align*}
w_0(x)&=k\mathscr{U}\big(x; p, q\big),\quad x\geqslant0, \\
w_1(x)&=\tilde{k}\Big[(1+\beta_0)\mathscr{U}\big(x; p, q\big)-2x\mathscr{U}\big(x; p, 1+ q\big)\Big],\quad x\geqslant0,
 \end{align*}
where
  \begin{align*}
   k&=\frac{\beta_0\Gamma(1-q)}{\Gamma(\alpha_1)\left[\beta_0\,{_2}F_1\big(1\, , 1+p\, ; 1-q\, ; -1\big)-q\right]} , \\
     \tilde{k}&=\frac{k}{2p},\ \lambda_0=\frac{\Gamma(\alpha_1)}{\beta_0\Gamma(-q)}k\ \ \mbox{and}\  \ \lambda_1=\frac{\Gamma(\alpha_1)}{\Gamma(-q)}\tilde{k}.\hskip 1cm
  \end{align*}
   The detailed verification of these results being left to the reader.\\

 \noindent We now return to the two limiting cases mentioned earlier. Regarding the first one, we find again a pair of weights with addition of Dirac masses at the origin. So the proof is only sketched.
  The second case can be treated in a similar way as in Proposition 4.2, and so the obtained results will be announced  without  proofs.\par
 \medskip
  \noindent  $\bullet$ The case $1+p+q=0\Leftrightarrow \alpha_1=1$, which gives $-1<\beta_0<1$. In this case we obtain that $u_0$ and $u_1$ admit
  the integral representations \eqref{eq:4.21}-\eqref{eq:4.22} with
\begin{align}
w_0(x)&=k \mathscr{U}\big(x;p,-p-1\big),\ \ x\geqslant0,\label{eq:4.51}\\
w_1(x)&={-{k}}\big[\mathscr{U}\big(x;p,-p-1)+\frac{2}{1+\beta_0}x\mathscr{U}\big(x;p,-p\big)\big],\ \ x\geqslant0,\label{eq:4.52}
\end{align}
satisfying  the conditions
  \begin{subequations}\begin{align}
\left[x w_0 f'-x w_0'f \right]_{\mathscr{C}}+\lambda_0\left[f'(0)+\beta_0f(0)\right]&=0, \label{eq53a}\\
\left[x w_0 f\right]_{\mathscr{C}}-(1+\beta_0)\lambda_1f(0)-\lambda_0\beta_0f(0)&= 0, \label{eq53b}
\end{align}\end{subequations}
where the constants $k$, $\lambda_0$ and $\lambda_1$  are such that
\begin{align*}
k=\frac{\beta_0\Gamma(2+p)}{1\!+\!p\!+\!\beta_0 \, {_2}F_1\big(1\, , 1+p\, ; 2+p\, ; -1\big)} ,\,
\lambda_0=\frac{k}{\beta_0\Gamma(1\!+\!p)}\, \mbox{and} \ \lambda_1=\frac{-{k}}{(1+\beta_0)\Gamma(\!1+\!p)}. 
\end{align*}
The proof of \eqref{eq:4.51}-\eqref{eq:4.52} runs as for \eqref{eq:4.43}-\eqref{eq:4.44} when $x$ lies in the positive real axis.
Finally, it remains to determine $\lambda_0$, $\lambda_1$ and $k$ so that \eqref{eq53a}-\eqref{eq53b} are met. \hfill\break
At first, we see  that $\!\lim\limits_{x\to+\infty}\!xw_0(x)\!=\!\lim\limits_{x\to+\infty}\!xw'_0(x)=0$. Next, due to Tables 1\&2, we easily check that
 $\underset{x\to 0^+}{\lim}xw_0(x)=0\  \mbox{and}\ \underset{x\to 0^+}{\lim}xw'_0(x)=-{k}/{\Gamma(1+p)}$. \\
On account of these results,  the system \eqref{eq53a}-\eqref{eq53b} simplifies to
\begin{align*}
kf(0)-\Gamma(1+p)\lambda_0\big(f'(0)+\beta_0f(0)\big)&=0,\\
(1+\beta_0)\lambda_1f(0)+\lambda_0\beta_0f(0)&= 0,
\end{align*}
which, for $f(x)=1$, implies $\lambda_0=k/\beta_0\Gamma(1+p)$ and $\lambda_1=-k/(1+\beta_0)\Gamma(1+p)$.\par
\noindent On applying now  \eqref{eq:4.21} for $f(x)=1$ and then use \eqref{eq:2.22}, we obtain the identity
\begin{align*}
\Big[1+p+\beta_0\, {_2}F_1\big(1\, , 1+p\, ; 2+p\, ; -1\big) \Big]k=\beta_0\Gamma(2+p).
\end{align*}
This allows us to derive the constant $k$ and, in turn, the expressions of $\lambda_0$ and $\lambda_1$ as stated above.\par
\noindent $\bullet$ The case $2+p+q=0\Leftrightarrow \alpha_1=2$. This case holds  if and only if $p=q=-1$, and consequently $\beta_0=0$.
Thus,   the differential system  \eqref{eq:4.38} and \eqref{eq:4.39}  becomes
 \begin{align}
 & w_0''- w_0 = 0 , \label{eq:4.54}\\
&2  w_1 = -xw_0'+\left(x-1\right)w_0,\label{eq:4.55}
\end{align}
provided that
  \begin{subequations}\begin{align}
\left[x w_0 f'+\left(w_0 -x w_0'\right)f \right]_{\mathscr{C}}&=0,\label{eq56a}\\
\left[x w_0 f\right]_{\mathscr{C}}&= 0. \label{eq56b}\end{align}
  \end{subequations}
  The general solution of Equation \eqref{eq:4.54} is $w_0(x)= c_1e^{x}+c_2e^{-x}$, for all $\, x\in\mathbb{R}$. So, for $\mathscr{C}=]-\infty\, ; +\infty[$,
  the pair  of weights $w_0$ and $w_1$, subject to \eqref{eq56a}-\eqref{eq56b}, may   be   expressed in the form
  \begin{align*}
w_0(x)\!=\!\left\{\begin{aligned}
&c e^{-x}, &&\mbox{if}\ x\geqslant0,\\
&c e^{x}, &&\mbox{if}\ x\leqslant0,
\end{aligned}\right.\ \mbox{and}\
w_1(x)\!=\!\left\{\begin{aligned}
&\!-c^2\left(1\!-2x\right)e^{-x}, &&\mbox{if}\ x\geqslant0,\\
&\!-c^2 e^{x}, &&\mbox{if}\ x\leqslant0,
\end{aligned}\right.\ \mbox{with}\ c=\frac{1}{2}.
\end{align*}
 \noindent{\bf - Subcase III.2}: For $\eta = 0$,  we fix $r=1$ and set $\beta_0=\alpha_1:=\alpha+1$ ($\alpha$   parameter).
From this, the two polynomials $\sigma$ and $\tau$ write $\sigma(x)=-x$ and $\tau(x)=-x+\alpha$.\\
  On substituting these expressions into   $({\cal S}_{3})$, we deduce that
 \begin{align}
  & xw_0'+(x-\alpha)w_0=0 ,\label{eq:4.57}\\
  & w_1'-w_1 =w_0',\label{eq:4.58}
\end{align}
under the boundary conditions \eqref{eq:4.28}-\eqref{eq:4.29} which we can  write
 \begin{subequations}
 \begin{align}
  \left[x w_0f\right]_{\mathscr{C}}-\lambda_0\big(1+\alpha\big)f(0)=0,\label{eq59a}\\
 \left[(w_1- w_0)f\right]_{\mathscr{C}}+(\lambda_1-\lambda_0)f'(0)+\lambda_1 f(0)=0.\label{eq59b}
 \end{align}
\end{subequations}
\begin{proposition}
For $\alpha\geqslant0$, $x\geqslant0$,  the  functionals $u_0$ and $u_1$ admit the integral representations \eqref{eq:4.21}-\eqref{eq:4.22}
with $\lambda_0= 0$,  $\lambda_1= -{1}/{2^{\alpha+1}}$, and the weights functions
\begin{align*}
  w_0(x)= \frac{1}{\Gamma(\alpha+1)}x^\alpha e^{-x}\ \ \mbox{and}\ \
  w_1(x)=  w_0(x)-\frac{e^x}{2^{\alpha+1}}{\frac{\Gamma(\alpha+1, 2x)}{\Gamma(\alpha+1)}}.
\end{align*}
   \end{proposition}
\begin{proof}
 The general solution of Equation \eqref{eq:4.57} may be written in the form
\begin{align*}w_0(x)=\left\{\begin{aligned}
  &c_1 x^\alpha e^{-x} ,   & &x\geqslant0,\\
  &c_2 |x|^\alpha e^{-x} , & & x<0,
\end{aligned}\right. \quad \mbox{for}\  \alpha>-1,
  \end{align*}
  where $c_1$, $c_2$ are two  arbitrary constants suitably chosen.
 Observe that   \eqref{eq59a} is realized if we take $c_2=0$ and $\lambda_0=0$.
We next apply the normalization condition $\left<u_0 \, ,\, 1\right>=1$ obtaining $c_1={1}/{\Gamma(\alpha+1)}$, and so  $w_0(x)=\frac{1}{\Gamma(\alpha+1)} x^\alpha e^{-x}$. This clearly shows that $u_0$ coincides with the classical Laguerre functional denoted sometimes by  ${\cal L}^{(\alpha)}$.
 Substitute now $w_0(x)$ by its expression in  \eqref{eq:4.57} and then integrate to get
 \begin{align}
 w_1(x)=c_3e^x+ e^x\int_0^x\!\! w_0'(t)e^{-t}dt,\quad x\geqslant0.\label{eq:4.60}
 \end{align}
 Perform integration by parts and use \eqref{eq:2.27} with \eqref{eq:2.28} to obtain
 \begin{align*}
  \int_0^x w_0'(t)e^{-t}dt= \frac{1}{\Gamma(\alpha+1)}t^\alpha e^{-2t}\Big]_0^x+  \frac{1}{\Gamma(\alpha+1)}\int_0^x t^\alpha 
 \end{align*}
 The restriction $\alpha>0$ implies that the first part vanishes at $x=0$. Thus,
 \begin{align}
  \int_0^x w_0'(t)e^{-t}dt = w_0(x)e^{-x}+\frac{1}{2^{\alpha+1}}\left[1-{\frac{\Gamma(\alpha+1 , 2x)}{\Gamma(\alpha+1)}}\right].
  \label{eq:4.61}
 \end{align}
If we substitute \eqref{eq:4.61} into \eqref{eq:4.60} and choose the constant   $c_3=-{1}/{2^{\alpha+1}}$, then
 \begin{align*}
  w_1(x)=w_0(x)-\frac{e^x}{2^{\alpha+1}}{\frac{\Gamma(\alpha+1, 2x)}{\Gamma(\alpha+1)}}, \ x\geqslant0.
 \end{align*}
From this, on account of \eqref{eq:2.29} with $\Gamma(a, 0)=\Gamma(a)$, it follows that
\begin{align*}
  \Big[w_1(x)-w_0(x)\Big]_0^{+\infty}=\frac{1}{2^{\alpha+1}}.
 \end{align*}
 If $\alpha$ assumes the value $0$, one can quickly verify that Eqs. \eqref{eq:4.57} and \eqref{eq:4.58} have the solutions
$w_0(x)=e^{-x}$ and $w_1(x)= c_4e^{x}+\frac{1}{2}e^{-x}$, respectively, where $c_4$ is an arbitrary constant.
Since we do not need the  general solution of \eqref{eq:4.58}, we will choose $c_4=0$, so that $w_1(x)$ writes  $w_1(x)=w_0(x)- \frac{1}{2}e^{-x}$.
Thus
$$\Big[w_1(x)-w_0(x)\Big]_0^{+\infty}=\frac{1}{2}.$$
To summarize, \eqref{eq59b} is realized if we set $\lambda_1=-\frac{1}{2^{\alpha+1}}, \, \alpha\geqslant 0$.
\end{proof}

\noindent{\bf IV.} For $\delta_1=0 ,\,\delta_0\ne0 \Leftrightarrow s= r\ne0$, with $\eta= 2r\alpha_1-{\gamma}$, we have  the polynomials\\
\indent $\ \phi(x)=\gamma,\, \vartheta(x)=2r(x-\beta_0)-\alpha_1,\, \chi(x)=-(x-\beta_0), \, \varrho(x)=2r(x-\beta_0)$.\\
\noindent Once again we will consider the two subcases $\eta\ne0$ and $\eta=0$ separately.\\
\smallskip
\noindent {\bf- Subcase IV.1}: For $\eta\ne0$, take $\gamma=1$ in  $({\cal S}_{1})$ to obtain that $w_0$ and $w_1$ satisfy
\begin{align}
&w_0''+\left(2rx-2r\beta_0-\alpha_1\right) w_0'+(-x+\beta_0+2r) w_0 = 0,\label{eq:4.62}\\
&(2r\alpha_1-1) w_1 = w_0'+2r(x-\beta_0)w_0,\label{eq:4.63}
\end{align}
under the boundary conditions
\begin{subequations}\begin{align}
  & \left[w_0 f'-\big(w_0'+(2rx-r^{-1}) w_0\big)f \right]_{\mathscr{C}}=0,\label{eq64a}\\
  & \big[w_0f\big]_{\mathscr{C}}=0,\label{eq64b}\end{align}
 \end{subequations}

\begin{proposition} For $r>0$ and $\mathscr{C}=\mathbb{R}$,  the  functionals $u_0$ and $u_1$ admit the integral representations \eqref{eq:4.6}-\eqref{eq:4.7}, with
\begin{align}
w_0(x)&=ke^{-rx^2+x/2r} U\big(\alpha\, ;  \frac{1}{2}\, ; rx^2\big), \label{eq:4.65}\\
w_1(x) &= \tilde{k}e^{-rx^2+x/2r}\Big[2\alpha rxU\big(\alpha+1\, ;  \frac{3}{2}\, ; rx^2\big)+\big(2r\beta_0-\frac{1}{2r}\big)U\big(\alpha\, ;  \frac{1}{2}\, ; rx^2\big)\Big],\label{eq:4.66}
\end{align}
where $\alpha$ is defined by $\alpha:= {\eta}/{(16r^3)}$, $k={1}/{J(\alpha\, ,r)}$ and $\,\tilde{k}=-k/\eta$.
\noindent  The notation $J(\alpha\, ,r)$ stands for the value of the integral
 \begin{align}
 \int_{-\infty}^{+\infty}\!e^{-rx^2+\frac{x}{2r}}U\big(\alpha\, ; \frac{1}{2}\, ;  rx^2\big)dx. \label{eq:4.67}
\end{align}
\end{proposition}

\begin{proof}
 To obtain a solution of Eq. \eqref{eq:4.62}, Laplace's method could be used again.
 However, in this case,  an analogous procedure does not give  an obvious way to evaluate the integrals that are involved in the solutions.
 Hence,  a different approach will be applied here. First, on substituting $w_0(x)=e^{x/2r}U(x)$ into \eqref{eq:4.62} we obtain
\begin{align*}
U''(x)+(c_0+b_0x) U'(x)+a_0 U(x) = 0,
\end{align*}
where  $\, a_0= 2r-{\alpha_1/2r}+{1/4r^2}$, $\; b_0=2r$ and  $c_0=-\alpha_1-2r\beta_0+{1/r}$.\\
Next, a change of variable $t=b_0x +c_0$ with the transformation $U(x)=V(t)$ easily sets up the equation
\begin{align*}
\hskip2cm V''(t)+b_1t V'(t)+a_1 V(t) = 0, \quad (\mbox{with}\ a_1={a_0/{4r^2}} \ \mbox{and} \ {b_1=1/2r}).
\end{align*}
 On substituting now $V(t)=W(z)$, with $z=-\frac{1}{4r}t^2$, into the last equation  we obtain that $W(z)$ satisfies
\begin{align*}
z W''(z)+(c-z) W'(z)-a W(z) = 0,
\end{align*}
which is actually a Kummer equation with $c=\frac{1}{2}$ and $a=\frac{1}{2}-\alpha$.\\
In order to use the same weight (or two weights)  on  two touching  paths, starting or ending at $0$, we are lead to consider the values
 of the desired solution in the immediate neighborhood of the origin. To this end, if we take $c_0 = 0$, then $z=-rx^2$.
 By the choice we have made, for $r>0$, a  solution of the above Kummer equation may be  taken in the form  $e^{-rx^2}U\big(\alpha\, ; \frac{1}{2}\, ; rx^2\big)$,
  so that the weight function $w_0$ writes  $w_0(x)=ke^{-rx^2+x/2r} U\big(\alpha\, ;  \frac{1}{2}\, ; rx^2\big)$, where $k$ is an arbitrary constant.\\
 It is known that the function $U(a\, ; c\, ;x)$ develops no discontinuities away from $x=0$, being real and finite for all values of the parameters
  and all positive arguments.\\
For $c=\frac{1}{2}<1$, referring to Table 1, it is easy to show that $w_0(x)$ has a finite value at $x=0$ because  we have
$ \underset{x\to 0^-}{\lim}w_0(x)= \underset{x\to 0^+}{\lim}w_0(x)=\Gamma(\frac{1}{2})/\Gamma\left(\frac{1}{2}+\alpha\right)$.\\
On the other hand, since $U(a\, ; c\, ;x)$ behaves like   a power as the argument approaches infinity, due to the presence of the factor $e^{-rx^2}$,
 we immediately conclude that $\lim\limits_{x\to-\infty}w_0(x)=\lim\limits_{x\to+\infty}w_0(x)=0$.\\
We next turn to evaluating the constants $k$. The normalization condition implies
 \begin{align*}
  1= \int_{-\infty}^{+\infty}w_0(x)dx=k \int_{-\infty}^{+\infty}\!\!e^{-rx^2+\frac{x}{2r}}U(\alpha\, ; \frac{1}{2}\, ;  rx^2)dx.
 \end{align*}
The integral in the right hand side actually coincides with the one  indicated  in \eqref{eq:4.67}. From what has already been shown, we can deduce that this integral is convergent but it is far being so easy to evaluate it. We will therefore limit ourselves to considering its supposed value to write  $k=1/ J\big(\alpha\, , r\big)$.\\
Our next concern will be the behaviour of $w'_0(x)$ as $x\to\pm\infty$. Differentiate $w_0(x)$, using \eqref{eq:2.18}, to obtain
\begin{align*}
w'_0(x)=ke^{-rx^2+x/2r}\Big[\big(\frac{1}{2r}-2rx\big)U\big(\alpha\, ;  \frac{1}{2}\, ; rx^2\big)-2\alpha rxU\big(\alpha+1\, ;  \frac{3}{2}\, ; rx^2\big)\Big],
\ x\in \mathbb{R}.
\end{align*}
This clearly shows that $\lim\limits_{x\to\pm\infty}w'_0(x)=0$.   Since $\lim\limits_{x\to\pm\infty}w_0(x)=0$, we conclude that  the vanishing conditions \eqref{eq64a}-\eqref{eq64b} are fulfilled.\\
Inputting finally  both $w'_0(x)$ and $w_0(x)$ by their expressions in Eq. \eqref{eq:4.63}, to  find the function $w_1$ given by \eqref{eq:4.66} with  $\tilde{k}=k/{\eta}$.\\
We conclude this proof by noting that both $w_0(x)$ and $w_1(x)$ are finite at $x=0$, which enables us to assert that the continuity of these functions
 holds on all $\mathbb{R}$.
\end{proof}

\noindent {\bf- Subcase IV.2}: For $\eta = 0$ and $\delta_0=2r$, we have  $\sigma(x)=-\alpha_1$ and $\tau(x)=-x+\beta_0$.\\
Let $\beta_0=0$, $\alpha_1={1}/{2}$ and set $\mu=\delta_0^{-1}$. With these choices, we readily see that the system $({\cal S}_{3})$ implies that
  \begin{align}
  & w_0'+2xw_0=0,\label{eq:4.68}\\
   & w_1'-\mu w_1 =\mu w_0'.\label{eq:4.69}
  \end{align}
with $w_0$ and $w_1$  subject to
  \begin{subequations}
      \begin{align}
 & \left[w_0f\right]_{\mathscr{C}_1}=0, \label{eq70a}\\
 &\left[(w_1- \mu w_0)f\right]_{\mathscr{C}_2}=0,\label{eq70b}
 \end{align}
   \end{subequations}
   where ${\mathscr{C}_1}=]-\infty\, ;+\infty[$ and ${\mathscr{C}_2}=[0_, ;+\infty[$.
  \begin{proposition} For $\mu<0$, the two functionals $u_0$ and $u_1$ admit the integral representations
  \begin{align*}
\big< u_0 \, ,\,f \big> &= \int_{-\infty}^{+\infty}\!\! f(x) w_0(x) dx  ,\ f\in{{\mathscr P}},\\
\big< u_1 \, ,\, f\big> &= \int_{0}^{+\infty} \!\! f(x) w_1(x)dx ,\ f\in{{\mathscr P}},
\end{align*}
where the weight functions $w_0$ and $w_1$ are, respectively, given by
\begin{align*}
  w_0(x)&= \frac{1}{ \sqrt{\pi}} e^{-x^2},\quad x\in\mathbb{R},\notag\\
  w_1(x)&=\mu  w_0(x) + {\frac{1}{2}} \mu^2 e^{\frac{\mu^2}{4}} \left[{\rm erf}\bigl(x+\frac{\mu}{2}\bigr)-{\rm erf}\bigl(\frac{\mu}{2}\bigr)\right]e^{\mu x},
  \quad x\geqslant0. 
\end{align*}
   \end{proposition}
\begin{proof}
 We can  proceed analogously to the proof of Proposition 4.3 with two slight differences. First, the  two weights $w_0$ and $w_1$ are supported on different intervals.  Second, the disappearance of Dirac's mass. We leave it to the reader to verify that the pair  $w_0$ and $w_1$ given above is the solution to the differential system \eqref{eq:4.68}-\eqref{eq:4.69}, under  the boundary conditions \eqref{eq70a}-\eqref{eq70b}. We conclude  with the observation  that  $w_0(x)$  coincides with the classical Hermite weight.
   \end{proof}

\noindent{\bf V.} For $\delta_1\ne0 ,\, \delta_0=0 \Leftrightarrow s=- r\ne0$, we have the polynomials\\
\indent $\phi(x)=2s\alpha_1+\gamma=-\eta, \ \vartheta(x)=2s(x-\beta_0)-\alpha_1,\ \chi(x)=-(x-\beta_0), \ \varrho(x)=0$.\\
Since  $\phi(x)\ne0$, the subcase $\eta=0$ cannot occur here as noted in \cite[Rem. 3.1]{1}.\\
\noindent- {\bf Subcase V.1}:  Setting  $\eta=-1$, we see that the system $({\cal S}_2)$ writes
\begin{align}
&w_0''+\big(2sx-2s\beta_0-\alpha_1\big) w_0'+(-x+\beta_0+2s)w_0 = 0,\label{eq:4.71} \\
& w_1 = -w_0',\label{eq:4.72}
\end{align}
under the boundary conditions
\begin{subequations} \begin{align}
 & \left[w_0 f'+\big(w_0'+(2sx-s^{-1}) w_0\big)f \right]_{\mathscr{C}}=0,\label{eq73a}\\
  & \big[w_0f\big]_{\mathscr{C}}=0.\label{eq73b}\end{align}
\end{subequations}

\begin{proposition}  For $s>0$ and $\mathscr{C}=\mathbb{R}$, the two weights $w_0$ and $w_1$ are given by
\begin{align*}
w_0(x)&=ke^{-sx^2+x/2s} U\big(\alpha\, ;  \frac{1}{2}\, ; sx^2\big), \\
w_1(x)&=ke^{-sx^2+x/2s}\left[2sxU\big(\alpha\, ;  \frac{3}{2}\, ; sx^2\big)-\frac{1}{2s} U\big(\alpha\, ;  \frac{1}{2}\, ; sx^2\big)\right],
\end{align*}
where $\alpha:={1}/{(16s^3)}$, $k={1}/{J(\alpha\, , s)}$ and $J(\alpha\, , s)$  the value of the integral \eqref{eq:4.67}.
\end{proposition}

\begin{proof}
 The proof is, for the most part, only sketched since similar reasoning to that used in the proof of Proposition 4.4 applies here with the roles of $r$ and $s$  reversed.
If we interchange $r$ and $s$, we easily see that the two differential equations \eqref{eq:4.62} and \eqref{eq:4.71} are actually identical.  The solution of the latter, satisfying \eqref{eq73a}-\eqref{eq73b}, is therefore drawn from  \eqref{eq:4.65} without extra work.  The solution  of \eqref{eq:4.72} is  immediately deduced from $w_0$ by working out the derivative and use of \eqref{eq:2.19}.
 \end{proof}

 \noindent{\bf VI.} The last case  occurs for  $\delta_1\delta_0\ne0$. We must certainly assume that $sr\ne0$, since otherwise we  recover Case
  {\bf II} or  {\bf III}. Unlike what we announced in \cite[p.18]{1} concerning  this case, we still consider the two subcases $\eta\ne0$ and $\eta=0$.\\
\noindent {\bf- Subcase VI.1}: For $\eta\ne0$, the system to be considered is $({\cal S}_{1})$ with \\
\indent $\,\phi(x)=-\delta_1\delta_0(x-\beta_0)+\delta_1\alpha_1 +{\gamma},\ \vartheta(x)=2s(x-\beta_0)-\alpha_1$, \hfill\break
\indent $\, \chi(x)=-x+\beta_0\ \mbox{and}\ \varrho(x)=\delta_0(x-\beta_0).$\hfill \break
\noindent Setting $\delta_1\delta_0=1$ (or equivalently $s^2 - r^2=1$) and taking $\delta_1\alpha_1+\beta_0+\gamma=0$, we move the singularity of the second order differential equation in $({\cal S}_{1})$ at the origin. We thus obtain
 \begin{align}
&x w_0''+(-2sx+2s\beta_0+\alpha_1+2) w_0'+ (x-\beta_0-2s) w_0 = 0,\label{eq:4.74}\\
&\eta w_1=-x w'_0 +\delta_0(x-\beta_0-\delta_1)w_0,\label{eq:4.75}
 \end{align}
with  $w_0$ and $w_1$  subject to the conditions \eqref{eq:4.8}-\eqref{eq:4.12} yielding
 \begin{subequations} \begin{align}
&\left[xw_0 f'-\left(xw_0'+(2s\beta_0+\alpha_1+1- 2sx)w_0\right)f \right]_{\mathscr{C}}=0, \quad \forall f \in{\mathscr P},\label{eq76a}\\
&\left[x w_0f\right]_{\mathscr{C}}=0 ,\quad \forall f\in {\mathscr P}.\label{eq76b}
 \end{align}\end{subequations}
For simplicity of notations, we set $p=\frac{\delta_0}{2r}\left(\delta_0\beta_0+\alpha_1\right)$ and $q=-\frac{\delta_1}{2r}\left(\delta_1\beta_0+\alpha_1\right)$
which we  can rewrite as $p=-\frac{\delta_0^2}{2r}\gamma\ \mbox{and}\ q=-\frac{\delta_1^2}{2r}\eta$. We now state the  proposition.

 \begin{proposition} Let $\delta_1<0<\delta_0$ and let $p>-1, \, q>-1$ with $p, q\notin \mathbb{Z}_0^+$. Then, the functionals $u_0$ and $u_1$ admit the integral representations \eqref{eq:4.6}-\eqref{eq:4.7}  with
\begin{align}
w_0(x)&\!=\!\left\{\begin{aligned}
&k_1e^{sx}\mathscr{U}\big(rx;q,p\big),  &&\mbox{if}\ x\geqslant0, \\
&k_2e^{sx}\mathscr{V}\big(rx;p,q\big),  &&\mbox{if}\ x\leqslant0 ;
\end{aligned}\right.\label{eq:4.77}\\
 w_1(x)&\!=\!\left\{\begin{aligned}
&\tilde{k}_1e^{sx}\!\left[2rx\mathscr{U}\!\big(rx;q,1+p\big)-(1+\delta_0\beta_0)\mathscr{U}\big(rx;q,p\big)\right],  \ \ \ \mbox{if}\ x\geqslant0, \\
&\tilde{k}_2e^{sx}\!\left[(2rx\!-\!\delta_0\beta_0\!-\!1)\mathscr{V}\!\big(rx;p,q\big)\!-\!2rx\mathscr{V}\!\big(rx;p,1\!+\!q\big)\right],  \ \mbox{if}\ x\leqslant0,
\end{aligned}\right.\label{eq:4.78}
\end{align}
where the constants  $k_i$ and $\tilde{k}_i$, $i=1, 2$, are to be chosen suitably later.
\end{proposition}

\begin{proof} The result which will proved is directly  derived by simple application of Laplace's method as described in Case III; the proof is then adapted from that of Proposition 4.2. Again, a solution of Eq. \eqref{eq:4.74} is to be found in  the integral form \eqref{eq:4.46} with the function $v$ and the contour ${\cal C}$ obey the two equations \eqref{eq:4.48}, where
\[P(t)=(2s\beta_0+\alpha_1+2)t-(\beta_0+ 2s)\ \mbox{and}\ Q(t)=t^2-2st+1=\left(t-\delta_0\right)\left(t-\delta_1\right).\]
Then it follows immediately that
$$v(t)=\left(t-\delta_0\right)^p\left(t-\delta_1\right)^q\quad \mbox{and}\quad \Big[\left(t-\delta_0\right)^{p+1}\left(t-\delta_1\right)^{q+1}e^{xt}\Big]_{\cal C}=0.$$
It is easy to check that  $\delta_1<0<\delta_0$ is equivalent to $r>0$ {and} $-r<s<r$.\\
Let us now consider the following three intervals
\begin{align*}
{\rm 1.}\ \ {\cal C}_1&=\,] -\infty\, ;\, \delta_1\,[, \ &&\mbox{if}\  x>0, \ q+1>0,\\
{\rm 2.}\ \ {\cal C}_2&=\,]\, \delta_0\, ;\, +\infty\,[, \  && \mbox{if}\ x<0, \ p+1>0,\\
{\rm 3.}\ \  {\cal C}_3&=\,]\, \delta_1~ ;~\delta_0\,[, \ && \forall x\in\mathbb{R},\ p+1>0,\ q+1>0. \hskip 4cm
 \end{align*}
 It is easily seen that the integrand vanishes identically on each interval ${\cal C}_i, i=1,2,3$.\\
 In consequence, three particular solutions of \eqref{eq:4.74}, denoted $w^i_0(x),\, i\in\{1, 3\}$,  may be given in the integral form as
\begin{align*}
w^i_0(x)&=c_i\int_{{\cal C}_i}e^{xt}\left(t-\delta_0\right)^p\left(t-\delta_1\right)^q dt, \quad i\in\{1, 3\}.\hspace{3cm}
\end{align*}
Such solutions occur  under the  required  conditions on $x$, $p$ and $q$, with $c_i,\ i\in\{1, 3\}$, are  constants.
By a change of variable in each integral, due to \eqref{eq:2.20}-\eqref{eq:2.21}, we get
\begin{align*}
w^1_0(x)&=k_1e^{\delta_1x}U\big(1+q\, ; 2+p+q\, ; 2rx\big),  && \mbox{if}\ x>0,\ q>-1,\\
w^2_0(x)&=k_2e^{\delta_0x}U\big(1+p\, ; 2+p+q\, ; -2rx\big),  && \mbox{if}\ x<0,\ p>-1,\\
 w^3_0(x)&=k_3e^{\delta_1x} M\big(1+q\, ; 2+p+q\, ; 2rx\big), &&\mbox{for}\ x\in \mathbb{R},  \ p>-1,\ q>-1,  \
\end{align*}
where  $k_i,\ i\in\{1, 3\}$ are new constants. \\
In virtue of the  behavior of the function $M$ as $x\to \pm\infty$, and on account of the presence of the factor $e^{\delta_1x}$, it is clair that $w^3_0(x)$  is unbounded as $x\to\pm\infty$. \\
However, $w^1_0(x)\to 0$ as $x\to +\infty$ and $w^2_0(x)\to 0$ as $x\to -\infty$, as is easy to check.
 \noindent Therefore the two functions $w^1_0$ and $w^2_0$ will be the ones of interest to us here.
As a result, depending on the order of $p$ and $q$,  a solution of Eq. \eqref{eq:4.74}  can be written as a composition of both functions
  $\mathscr{U}$ and $\mathscr{V}$ as set in \eqref{eq:4.77}.
The continuity of $w_0(x)$ at $x=0$ is  ensured by considering the equality $w^1_0(0^+)=w^2_0(0^-)$, which is possible.\\
Working under the additional condition $1+p+q<0$ and using Table 1, we get
\begin{align*}
w_0^1(0^+)=k_1\frac{\Gamma(-1-p-q)}{\Gamma(-p)}\ \ \mbox{and}\ \ w_0^2(0^-)=k_2\frac{\Gamma(-1-p-q)}{\Gamma(-q)},\ \
 \mbox{for}\ p, q \notin \mathbb{Z}_0^+.
\end{align*}
So, for the above equality to happen, $k_1$ and $k_2$ must satisfy
\begin{align*}
\Gamma(-q)k_1-\Gamma(-p)k_2=0.
\end{align*}
Application of the standard normalization condition now leads to
\begin{align*}
 J_1(p,q) k_1+ J_2(p,q) k_2=1,
\end{align*}
where $J_1(p,q)$ and $J_2(p,q)$ are  two integrals, respectively, defined by
\begin{align*}
 J_1(p,q)&=\!\int_0^{+\infty}\!\!e^{\delta_1x}U\big(1+q\, ; 2+p+q\, ; 2rx\big)dx, \\
 J_2(p,q)&=\!\int_0^{+\infty}\!\!e^{-\delta_0x}U\big(1+p\, ; 2+p+q\, ; 2rx\big)dx.
\end{align*}
After a change of variable and  use of \eqref{eq:2.22} we get
\begin{align*}
J_1(p,q)&=-\frac{\Gamma(-p-q)}{\Gamma(1-q)}\delta_0\, {_2}F_1\big(1\, , 1+q\, ; 1-p\, ; \delta_0^2\big),\\
J_2(p,q)&=\frac{\Gamma(-p-q)}{\Gamma(1-p)}\delta_1 \, {_2}F_1\big(1\, , 1+p\, , ; 1-q\, ; \delta_1^2\big).
\end{align*}
From what has already been shown,  one sees immediately that
\begin{align*}
k_1\!=\!\frac{\Gamma(-p)}{\Gamma(-p)J_1(p,q)\!+\!\Gamma(-q)J_2(p,q)}\ \mbox{and}\ k_2\!=\!\frac{\Gamma(-q)}{\Gamma(-p)J_1(p,q)\!+\!\Gamma(-q)J_2(p,q)}.
\end{align*}
In order to obtain the second weight $w_1$, we start by computing $w'_0(x)$ using \eqref{eq:2.19}. On substituting now $w'_0(x)$ and $w_0(x)$ by their expressions  into  \eqref{eq:4.75}, we obtain \eqref{eq:4.78} with $\tilde{k}_1=k_1/\eta$ and $\tilde{k}_2=k_2/\eta$.
The continuity of $w_1$  at $x=0$ is established by using similar reasoning as in Proposition 4.2. For this purpose, we rely on the  connection between $k_1$ and $k_2$ and on  Table 2 obtaining
$\underset{x\to 0^-}{\lim}w_1(x)=\underset{x\to 0^+}{\lim}w_1(x)$, which is the desired result.
\end{proof}
Turning now to the proof of the two limiting cases that reappear here again. We note first that if we combine the restriction $1+p+q<0$ with those previously
 imposed on each of $p$ and $q$ we readily find that $-2<p+q<-1$.\\
 Although  the values of $p$ and $q$ here and  in Case III are distinct, the demonstrations are similar in spirit. And so,
    the related results will be announced without proofs. \\
\noindent $\bullet$ For  $1 + p + q = 0$, it is fairly easy to see that $p$ and $q$ are each lie in $]-1\, ; 0[$.
Here, the measures  are given via   two   weights with  add of Dirac masses at the origin:
\begin{align*}
 w_0(x)&=ke^{\delta_1x}U\big(-p\, ; 1\, ; 2rx\big),\ \ x\geqslant0,\\
 w_1(x)&= \tilde{k}e^{\delta_1x}\left[2rxU\big(-p\, ; 2\, ; 2rx\big)-(1+\delta_0\beta_0)U\big(-p\, ; 1\, ; 2rx\big)\right],\ \ x\geqslant0,
\end{align*}
satisfying the boundary  conditions
\begin{align*}
&\left[xw_0 f'-x\left(w_0'-2sw_0\right)f \right]_0^{+\infty}+\lambda_0\left(f'(0)-\beta_0f(0)\right)=0, \quad \forall f \in {{\mathscr P}},\\
&\left[x w_0f\right]_0^{+\infty}-\eta\lambda_1f(0)-\delta_0\beta_0\lambda_0 f(0)=0 ,\quad \forall f\in {{\mathscr P}},
 \end{align*}
which, combined to the normalization condition, enable us to write
\begin{align*}
 k&=\frac{\beta_0\Gamma(1-p)}{p-\beta_0 \delta_0\,{_2}F_1\big(1\, ,-p\, ; 1-p\, ; \delta_0^2\big)},\ \tilde{k}=\frac{ \delta_1^2 k}{2r(1+p)}, \\
 \lambda_0&=-\frac{k}{\beta_0\Gamma(-p)}\ \mbox{and} \ \lambda_1= \frac{\delta_1 k}{2r(1+p)\Gamma(-p)}.
\end{align*}
 $\bullet$ For $2 + p + q = 0$, an easy computation shows that $p=q=-1$, which in turn leads to  $\beta_0=-2s$, $\alpha_1=2(s^2+r^2)$ and $\eta=2r/\delta_1^2$.
  Thus, Eqs. \eqref{eq:4.74}-\eqref{eq:4.75} yield
\begin{align*} &w_0''-2s w_0'+   w_0 = 0,\\
&\eta w_1=-x w'_0 +\delta_0(x-\beta_0-\delta_1)w_0,
 \end{align*}
 under the related boundary conditions resulting from  \eqref{eq76a}-\eqref{eq76b}.
The general solution of the first equation is given by $w_0(x)= c_1e^{\delta_1x}+c_2e^{\delta_0x}$ which is valid for all $x\in \mathbb{R}$, with $c_1$, $c_2$
are two arbitrary constants.\\
There are three possibilities to provide integral representations for  $u_0$ and $u_1$.\\
First, by means of weights  supported on ${\mathscr{C}}=]-\infty\, ;+\infty[$ for which  the related results will be announced without proof.
In the two others, namely,  for the supports ${\mathscr{C}}=[0\, ;+\infty[$ or ${\mathscr{C}}=]-\infty\, ;0]$,   the measures are given via two weights with adding  Dirac masses at $x=0$. These alternatives are left to the reader.\\
For ${\mathscr{C}}=]-\infty\, ;+\infty[$, as in Case III, we quickly obtain the two weights
\begin{align*}
w_0(x)=\left\{\begin{aligned}
&c e^{\delta_1x}, &&\mbox{if}\ x\geqslant0,\\
&c e^{\delta_0x}, &&\mbox{if}\ x\leqslant0,
\end{aligned}\right.\quad\mbox{and}\quad
w_1(x)=\left\{\begin{aligned}
&-c^2\left(1+2r\delta_1^2x\right)e^{\delta_1x}, &&\mbox{if}\ x\geqslant0,\\
&-c^2 e^{\delta_0x}, &&\mbox{if}\ x\leqslant0,
\end{aligned}\right.
\end{align*}
 where  $c=- {1}/{(2r)}$. The boundary conditions  are evidently met.\par

\noindent {\bf - Subcase VI.2}: For $\eta = 0$, $\delta_1\delta_0\ne0$ and $sr\ne0$, the system to be used is once again $({\cal S}_{3})$ with the polynomials $\sigma(x)=\delta_1(x-\beta_0)-\alpha_1$ and $ \tau(x)=-x+\beta_0+\delta_1$.
We choose $\delta_1=-1(\Rightarrow \delta_0=2r-1)$,   $\beta_0=\alpha_1:=\alpha+1$ and set $\nu=\delta_0^{-1}$ to obtain
\begin{align}
& xw_0'+(x-\alpha)w_0=0,\label{eq:4.79}\\
& w_1'-\nu w_1 = \nu w_0',\label{eq:4.80}
\end{align}
with $w_0$ and $w_1$  subject to
\begin{subequations} \begin{align}
  & \left[x w_0f \right]_{\mathscr{C}}=0,\label{eq81a}\\
  & \left[(w_1-\nu w_0)f\right]_{\mathscr{C}}=0.\label{eq81b}
\end{align}\end{subequations}

\begin{proposition} Let $-1<\nu<0$. For $\alpha\geqslant0$  and $x\geqslant0$, the two functionals $u_0$ and $u_1$ admit the integral representations
\eqref{eq:4.21}-\eqref{eq:4.22} with $\lambda_0=\lambda_1= 0$, where
\begin{align*}
 w_0(x)=\frac{1}{\Gamma(\alpha+1)}x^\alpha e^{-x}\quad\mbox{and}\quad
 w_1(x)=\nu w_0(x) -{\frac{\nu^2 e^{\nu x}}{(\nu+1)^{\alpha+1}}} {\frac{\gamma\big(\alpha+1,(\nu+1)x\big)}{\Gamma(\alpha+1)}}.
 \end{align*}
    \end{proposition}
    \par

\begin{proof}
  Observe that Eqs. \eqref{eq:4.79} and  \eqref{eq:4.57} are identical.  So, analysis similar to that in the proof of Proposition 4.3 shows that  they have the same solution which is the classical Laguerre's weight function. On substituting  this function  into \eqref{eq:4.80} and then perform integration by parts we obtain the weight function $w_1(x)$ which holds for  $\alpha\geqslant0$. In the same manner we can see that, under the conditions stated above, the vanishing conditions \eqref{eq81a}-\eqref{eq81b} are fulfilled.
  \end{proof}
\noindent {\bf Remark 4.3.} It is worth noting that  the recurrence coefficients given by the expressions \eqref{eq:1.3}-\eqref{eq:1.5} simply provide particular solutions to
the system (2.16)-(2.21) established in \cite{1}.
The study we have carried  in the two parts of this work was done under the conditions $\theta_n=1$ and $\rho_n=1$.  This case  was denoted  {\bf (A-i)}.\\
It remains however an open question if there is another particular solution for the same system when  assuming that $\theta_n=1$ and $\rho_n=\frac{n+\rho+1}{n+\rho}, n\geqslant0,\ \rho\notin \mathbb{Z}_-$ (referred to as Case {\bf (A-ii)}).
Currently, we are not able to answer affirmatively to this issue. \par

\end{document}